\documentclass{article}

\usepackage{amsfonts}
\usepackage{amsmath}
\usepackage{a4wide}
\usepackage{amssymb}
\usepackage{hyperref}
\usepackage{tikz-cd}
\usetikzlibrary{cd}

\newtheorem{lem}{Lemma}[section]
\newtheorem{theo}[lem]{Theorem}
\newtheorem{coro}[lem]{Corollary}
\newtheorem{propo}[lem]{Proposition}
\newtheorem{rema}[lem]{Remark}
\newtheorem{defi}[lem]{Definition}

\newenvironment{proof}{\paragraph*{Proof}}
{\par}

\newcommand\GL{{\mathrm{GL}}}
\newcommand\SL{{\mathrm{SL}}}

\newcommand\eps\varepsilon
\newcommand\ph\varphi
\newcommand\C{{\mathbb C}}
\newcommand\F{{\mathbb F}}
\newcommand\Q{{\mathbb Q}}

\newcommand\PPP{{\mathbb P}}

\newcommand\Z{{\mathbb Z}}

\newcommand\E{{\mathbb E}}

\newcommand{\univ}{\mathrm{univ}}

\title{Semistable reduction of modular curves associated with maximal subgroups in prime level}

\author{B. Edixhoven and P. Parent}

4

\begin{document}

\maketitle 
\begin{abstract}
  We complete the description of semistable models for modular curves
  associated with maximal subgroups of $\GL_2 (\F_p )$ (for $p$ any
  prime, $p>5$). That is, in the new cases of non-split Cartan modular
  curves and exceptional subgroups, we identify the irreducible
  components and singularities of the reduction mod~$p$, and the
  complete local rings at the singularities. We review the case of
  split Cartan modular curves. This description suffices for computing
  the group of connected components of the fibre at~$p$ of the N\'eron
  model of the Jacobian. 

\medskip

AMS 2010 Mathematics Subject Classification  11G18 (primary), 11G20, 14G35 (secondary). 
\end{abstract}

\tableofcontents

\section{Introduction}
\label{Introduction}

Let $p$ be a prime number. The picture given in Figure~\ref{figureX_0}
\begin{figure}
\begin{center}
\begin{picture}(320,200)(-30,-80)


\put(128,149){$w_p$}

\qbezier(120,140)(130,145)(140,140)
\put(120,140){\vector(-2,-1){4}} 
\put(140,140){\vector(2,-1){4}}


\qbezier(90,135)(165,110)(130,90)
\qbezier(130,90)(100,75)(130,60)
\qbezier(130,60)(170,45)(130,30)
\qbezier(130,30)(120,25)(115,20)

\qbezier(170,135)(95,110)(130,90)
\qbezier(130,90)(160,75)(130,60)
\qbezier(130,60)(100,45)(130,30)
\qbezier(130,30)(140,25)(145,20)



\qbezier(130,-25)(120,-20)(115,-15)
\qbezier(90,-70)(165,-45)(130,-25)

\qbezier(130,-25)(140,-20)(145,-15)
\qbezier(170,-70)(95,-45)(130,-25)

\put(75,-85){$\PPP^1$}
\put(176,-85){$\PPP^1$}


\put(147,10){$\cdot$}
\put(147,2){$\cdot$}
\put(147,-6){$\cdot$}

\put(113,10){$\cdot$}
\put(113,2){$\cdot$}
\put(113,-6){$\cdot$}


\put(40,3){$g(X_0 (p)) +1$}

\put(90,15){\vector(1,2){18}} 
\put(90,13){\vector(2,1){19}}
\put(90,0){\vector(1,-1){20}}

\put(90,24){\vector(1,3){24}}

\put(109,63){$\cdot$}
\put(109,60){$\cdot$}
\put(109,57){$\cdot$}


 
  

 
\end{picture}
\end{center}
\caption{${X}_0 (p)_{/ \overline{\F}_p}$}
\label{figureX_0}
\end{figure}
of the geometric special fiber of the stable model of $X_0 (p)$ over
$\Z_p$ now looks familiar to many number-theorists\footnote{To be
  completely correct, when $p\leq 19$ that model is only
  semistable.}. It has been described in the work~\cite{DR73} of
Deligne and Rapoport, and was actually known, in a slightly different
guise, by Kronecker. Having such a model at hand has proven crucial in many
questions -- not only for direct applications such as the computation
of semistable N\'eron models of the jacobian $J_0 (p)$ but also in
diophantine issues, such as the determination of the non-cuspidal
rational points of $X_0 (p)$ in Mazur's famous works \cite{Ma76}
and~\cite{Ma78}.

It is actually under similar motivations that we describe here a
semistable model, over a suitable extension of~$\Z_p$, of the modular
curve~$X_{\mathrm{ns}}^+ (p)$ attached to the normaliser of a
non-split Cartan subgroup in~$\GL_2(\F_p)$. Recently indeed
J.~Balakrishnan and her coauthors managed to elaborate on the
Chabauty-Kim method and prove that the modular curve
$X_{\mathrm{ns}}^+ (13)$ had only the expected trivial rational points
(see~\cite{BM17}). That constituted a tour-de-force, as the latter
curve had so far resisted all known methods on Earth. Their strategy
needs at some point a bit of knowledge of the reduction type of the
curve under study, and that knowledge was available because
$X_{\mathrm{ns}}^+ (13)$ is isomorphic to $X_{\mathrm{s}}^+ (13)$,
attached to the normaliser of a split Cartan, see~\cite{Ba14}, and for
that latter curve the necessary information was already available
from~\cite{Edix89}. For $p>13$, there is no isomorphism between split
and non-split Cartan curves, so our models for $X_{\mathrm{ns}}^+ (p)$
shall prove necessary for applying the quadratic Chabauty method
of~\cite{BM17} to the latter curves.

A bit more generally, we describe stable models of modular curves
associated with all maximal subgroups of~$\GL_2 (\F_p )$. One
classically knows (see e.g. \cite{Ma76}) that those subgroups (up to
conjugation) are the Borel, the normalizer of split and non-split
Cartan (defining the curves denoted by $X_{\mathrm{s}}^+ (p)$ and
$X_{\mathrm{ns}}^+ (p)$ respectively) and some exceptional subgroups,
which are lifts of the permutation groups~$\mathfrak{A}_4$,
$\mathfrak{S}_4$ or~$\mathfrak{A}_5$ in ${\mathrm{PGL}}_2 (\F_p
)$. Among those, note that the curve $X_{\mathrm{s}}^+ (p)$ is isomorphic to $X_0 (p^2 )/w_{p^2}$,
and that the case $X_0 (p^2 )$ had already been treated in the
article~\cite{Edix89} by the first-named author of the present work
(see also \cite{Edix89b} and \cite{Edix01}). Nevertheless we do our own computations in Section~\ref{SplitCartan}
below, and we treat the case of $X_{\mathrm{s}}^+ (p)$.

The newest part of the present study however is a complete description of fibre at~$p$ of
the stable model for the non-split Cartan curve
$X_{\mathrm{ns}}^+ (p)$ and the thickness of its singularities
(cf. Section~\ref{NonSplitCartan}). (Recall the {\it thickness} of a semi-stable curve over a complete discrete 
valuation ring~$R$ with uniformiser~$\pi$ and with separably closed residue field~$k$, at a singular 
point of the special fibre, is the unique natural number $n\geq 1$ such that the completed local ring 
is isomorphic to $R[[x,y]]/(xy-\pi^n)$ (cf. Definition~10.3.23 of~\cite{Liu02}). That $n$ is equal to $1$ plus the number of projective lines over $k$ that 
appear in the minimal resolution of the singularity. In the terminology of rigid geometry, the meaning of $n$ is 
that the tube of the singular point is the open annulus with inner radius $|\pi|^n$ and outer radius~$|1|$.)

   The case of exceptional groups is
probably of lesser interest. From the diophantine point of view, for
instance, Serre remarked that a simple argument on the action of inertia at $p$ 
in the mod~$p$ Galois representations attached to elliptic curves shows that
the modular curves associated with exceptional groups have no local
points with values in (not too ramified) $p$-adic fields, as soon as $p$ is large
enough, cf.~\cite{Ma77}, p.~118.  We however compute semistable
models for those modular curves in Section~\ref{Exceptional}.

Our method is first to describe stable models for the curves
${\overline{\cal M}} (\Gamma (p), {\cal P})$ associated with the full
level-$p$ structure, enhanced by some additional (finite, \'etale,
representable) moduli problem~${\cal P}$ over $\Z_p$. This is what we
do in Section~\ref{X(p)}, essentially following the
unpublished~\cite{Edix01}. We then take quotients by relevant
subgroups of $\GL_2 (\F_p )$, starting with the normalizer of
non-split Cartan. The fact that we added a level structure $\cal
P$ allows us to keep working with a fine moduli space. Finally we
assume $\cal P$ is Galois with group $G$, and taking the quotient by
$G$ yields a stable model for the coarse curve $X_{\mathrm{ns}}^+
(p)$.  We repeat that process for the split-Cartan curve and the exceptional subgroups.

We must mention that this approach is not well-suited to deal
with cases of level divisible by powers $p^r$ of $p$, when $r\geq 2$,
because of algebro-geometric reasons recalled in
Remark~\ref{higerp^n}. In that situation, probably, one can apply  
J.~Weinstein's results (see~\cite{Wein16}). It is however not clear to us
if those techniques will provide the thicknesses of the singularities of
the stable models, and how difficult it would be to find the graphs.

\medskip

A last word about stability versus semistability: as for the model of
$X_0 (p)$ recalled in Figure~\ref{figureX_0}, our semistable models
will actually be stable, for large enough $p$, in many cases but not
all.  The curves $X_{\mathrm{s}}^+ (p)$ and $X_{\mathrm{ns}}^+ (p)$,
for any $p$ which is $-1$ mod $4$, are indeed not stable, as explained
in Theorems~\ref{NS0+}, \ref{S0+}, and Remark~\ref{seulementsemi}. In
all cases however it is easy to spot what projective lines need to be
contracted in order to obtain a stable model. About that issue, see
Remark~\ref{level13}.


\section{Stable model for full level $p$ structure}
\label{X(p)}

\subsection{Katz-Mazur's model $\overline{{\cal M}} ({\cal P}, \Gamma (p))$}

Our starting point will be the modular model over~$\Z [\zeta_p ]$,  as given by Katz and Mazur
(\cite{KM85}; Chapter~13), for modular curves with full level $p$
structure plus some additional level structure $\cal P$ with nice
properties at~$p$. Let us very briefly recall Katz-Mazur
``Drinfeldian approach'' to moduli problems. We will not discuss
stable models for the curves $X(p)_\Q$ with no additional level
structure. 

\medskip

We let $\cal P$ be a representable finite \'etale moduli problem over
$(\mathrm{Ell})_{\Z_p}$. One can take for instance ${\cal
  P}=[\Gamma_1(N)]$ for $N\geq 4$ a prime-to-$p$ integer. Later, when
we want to get rid of~$\cal P$, we
will assume moreover that $\cal P$ is Galois over~$(\mathrm{Ell})_{\Z_p}$.

\medskip

There was a time when for $N$ any positive integer, $\Gamma (N)$
denoted the kernel of the reduction morphism $\SL_2(\Z)\to
\SL_2(\Z/N\Z)$. But since~\cite{DR73} it became clear that it was
better to attach modular curves to compact open subgroups of the
finite ad\`ele group $\GL_2(\Q\otimes\hat{\Z})$. So, we let
$\Gamma(N)$ denote the kernel of the surjective morphism
$\GL_2(\hat{\Z})\to \GL_2(\Z/N\Z)$. Following~\cite{KM85}, if $E/S$ is
an elliptic curve over an arbitrary scheme~$S$, we say that a group
morphism $\phi \colon (\Z /N \Z)^2 \to E(S)$ is a {\it $\Gamma
  (N)$-structure} (or {\it ``full level-$N$ structure''}) if the
effective Cartier divisor
$$
D_N :=\sum_{a\in (\Z /N \Z)^2} [\phi (a)]
$$
is a group scheme which is equal to $E[N]$. The ordered pair $(P:=\phi
(1,0), Q:=\phi (0,1))$ is then said to be a {\it Drinfeld basis}
of~$E[N]$. The set of $\Gamma(N)$-structures on $E/S$ is
denoted~$[\Gamma(N)](E/S)$.

\medskip

  Of course when $N$ is invertible on $S$, this notion of level-$N$ structure brings nothing new to the naive usual definition. On the other hand, over a field $k$ of positive characteristic $p$, 
a Drinfeld basis of $E[p]$ is easily seen to be a pair $(P,Q)$ such that at least one of the two points has order $p$, in the usual sense, in $E(k)$, at least if $E$ is ordinary (and  the only possible
$(0,0)$ if $E$ is supersingular). 

\medskip

Let us fix from now on some prime number $p$. If $S$ is an $\F_p$-scheme, and $n$ any non-negative integer, let $F_{E/S}^n \colon E \to E^{(p^n )}$ denote the $n^{\mathrm{th}}$-power
of the relative Frobenius, and $V_{E/S}^n \colon E^{(p^n )} \to E$ the $n^{\mathrm{th}}$-power of the Verschiebung, that is, the dual isogeny to $F_{E/S}^n$. 

One knows that $F_{E/S}^n$ is purely radicial, and $V_{E/S}^n$ is
\'etale exactly when $E/S$ is ordinary. In any case both isogenies are
cyclic with order $p^n$, that is, after a suitable surjective finite
locally free base change, there is a group morphism $\phi \colon
(\Z/p^n \Z )\to E(S)$ such that their kernel is equal, as effective
Cartier divisor, to $\sum_{a\in (\Z /p^n \Z)} [\phi (a)]$. 
An {\it
  Igusa structure of level $p^n$} on $E/S$ is the datum of some point
$P$ in $E^{(p^n )}[p^n](S)$ such that the equality
$$
{\mathrm{Ker}} (V_{E/S}^n )=\sum_{a\in (\Z /p^n \Z)} [aP]
$$ 
between effective Cartier divisors holds. The associated moduli
problem on the elliptic stack $({\mathrm{Ell}} )_{\F_p}$ is denoted by
$[{\mathrm{Ig}} (p^n )]$. Igusa proved that $[{\mathrm{Ig}} (p^n )]$
is relatively representable: there is a complete smooth curve $\overline{{\cal M}}
({\cal P}, [{\mathrm{Ig}} (p^n )])$ over $\F_p$ 
such that the complement of the cusps ${\cal M}({\cal P},
[{\mathrm{Ig}} (p^n )])$ represents $({\cal P}, [{\mathrm{Ig}} (p^n)])$.  
To state Katz-Mazur's central result
in the simplest way, we shall actually work with so-called {\it exotic Igusa structures}. So - restricting ourselves to level $p$ - we define the moduli problem:
$$
(E/S/\F_p) \mapsto 
[{\mathrm{ExIg}} (p,1 )](E/S) :=\{ P\in E(S), (0,P) {\mathrm{\ is\ a\ Drinfeld\ }} p{\mathrm{-basis\ of\ }} E/S  \}
$$
and we then can check there is an (exotic) isomorphism 
\[
[{\mathrm{Ig}} (p )] \stackrel{\sim}{\longrightarrow}[{\mathrm{ExIg}}(p,1 )]
\,,\quad
(E/S, P\in [{\mathrm{Ig}} (p )](E/S)) \mapsto (E^{(p)}/S,(0,P))\,.
\]

\medskip

The moduli problem $({\cal P},[\Gamma (p)])$ classifies triples
$(E/S,\alpha,\phi)$ for $S$ a $\Z_p$-scheme, $E/S$ an elliptic curve,
$\alpha\in{\cal P} (E/S)$, and $\phi\in[\Gamma (p)](E/S)$.
Katz-Mazur's theorems about $\Gamma(p)$-structures (\cite{KM85},
Theorems~3.6.0, 5.1.1 and 10.9.1) then assert that
$({\cal P},[\Gamma (p)])$ is representable by a regular $\Z_p$-scheme ${\cal
  M} ({\cal P} ,[\Gamma (p)])$, that has a compactification
$\overline{{\cal M}} ({\cal P} ,[\Gamma (p)])$ which enjoys the
following properties. Weil's pairing $e_p (\cdot ,\cdot )$ shows that
the morphism $\overline{\cal M} ({\cal P} ,[\Gamma (p)]) \to
{\mathrm{Spec}} (\Z_p)$ factorizes through ${\mathrm{Spec}}
(\Z_p[\zeta_p])$, with
$\Z_p[\zeta_p]:=\Z_p[x]/(x^{p-1}+\cdots+x+1)$. For all integers $i$
in $\{ 1,\dots ,p-1 \}$, set 
\[
X_i :=\overline{\cal M} ({\cal P} ,[\Gamma (p)^{\zeta_p^i{\mathrm{-can}} } ])
\]
for the sub-moduli problem over $({\mathrm{Ell}} )_{\Z_p[\zeta_p ]}$
representing triples $(E/S/\Z_p[\zeta_p],\alpha, \phi)$ such that
\[
e_p (\phi (1,0) ,\phi (0,1)) =\zeta_p^i\,.
\]  
The obvious morphism:
$$
\coprod_i X_i \to \overline{{\cal M}} ({\cal P} ,[\Gamma (p)])_{\Z_p[\zeta_p ]}
$$
induces, by normalization, an isomorphism of schemes over
$\Z_p[\zeta_p]$:
\begin{eqnarray}
\label{X_i}
\coprod_i X_i \stackrel{\sim}{\longrightarrow} 
\overline{{\cal M}} ({\cal P} ,[\Gamma (p)])^\sim_{\Z_p[\zeta_p ]}\,,
\end{eqnarray}
with $\overline{{\cal M}} ({\cal P} ,[\Gamma (p)])^\sim_{\Z_p[\zeta_p
  ]}\to \overline{{\cal M}} ({\cal P} ,[\Gamma (p)])_{\Z_p[\zeta_p ]}$
the normalization. The triviality of $p^{\mathrm{th}}$-roots  
of unity in characteristic~$p$ shows that, after the base change
$\Z_p[\zeta_p]\to\F_p$, the $X_{i,\F_p}$ are not only isomorphic to
each other but actually equal. Moreover, the modular interpretation of
a $\Gamma (p)$-structure $\phi \colon (\Z /p\Z )^2 \to E(k)$, in the
generic case of an ordinary elliptic curve $E$ over a field $k$ of
characteristic $p$, amounts to choosing some line $L$ in $(\Z /p\Z
)^2$ that plays the role of ${\mathrm{Ker}} (\phi )$, then some point
$P$ in $E(k)$ which defines the induced isomorphism $(\Z /p\Z )^2 /L
\stackrel{\sim}{\longrightarrow} E[p](k)$.  Making that into a proof,
Katz and Mazur give the following theorem.
\begin{theo}
\label{KatzMazur1}
{\bf (Katz-Mazur \cite{KM85}, 13.7.6).} Each curve $X_{i,\F_p}$
obtained from $X_i$ over $\Z_p[\zeta_p]$ via $\Z_p[\zeta_p]\to\F_p$, is
the disjoint union, with crossings at the supersingular points, of
$p+1$ copies of the $\overline{\cal M} ({\cal P})$-schemes
$\overline{\cal M}({\cal P}, [{\mathrm{ExIg}} (p,1 )])$
(cf. Figure~\ref{FigureX_i}). 
We label those Igusa schemes as
${\mathrm{Ig}}_{i,L}$ for $(i,L)$ running through $\F_p^\times \times
\PPP^1 (\F_p )$.
\end{theo}
%


\begin{rema}
\label{semistableouquoi}
{\rm One would like to think of the copies of the scheme $\overline{\cal M}({\cal P}, [{\mathrm{ExIg}} (p,1 )])$ as the ``components''
of $X_{i,\F_p}$, which is morally true - note however that they may not be geometrically irreducible (being such exactly when $\overline{\cal M}({\cal P} )$ is).
The same administrative issue will show up in our subsequent models. It of course vanishes when we eventually get
rid of the auxiliary level-${\cal P}$ structure, as in the coarse curves $X_{\mathrm{ns}} (p)$, $X_{\mathrm{ns}}^+ (p)$, and so on below.}
\end{rema}

\begin{figure}
\begin{center}
\begin{picture}(170,200)(50,-80)


\put(52,159){$p+1$ (vertical) copies of}
\put(61,147){$\overline{\cal M}({\cal P}, [{\mathrm{ExIg}} (p,1 )])$}

\put(40,150){\vector(-2,-1){20}}

\put(175,155){\vector(3,-2){20}}

 
\qbezier(-30,135)(45,110)(10,90)
\qbezier(10,90)(-20,75)(10,60)
\qbezier(10,60)(40,45)(10,30)
\qbezier(10,30)(0,25)(-5,20)

\qbezier(50,135)(-25,110)(10,90)
\qbezier(10,90)(40,75)(10,60)
\qbezier(10,60)(-20,45)(10,30)
\qbezier(10,30)(20,25)(25,20)

\put(0,130){$\cdot$}
\put(9,130){$\cdot$}
\put(18,130){$\cdot$}

\put(4,101){$\cdots$}

\put(5,74){$\cdots$}

\put(5,44){$\cdots$}

 
\qbezier(-30,130)(55,110)(10,90)
\qbezier(10,90)(-30,75)(10,60)
\qbezier(10,60)(50,45)(10,30)
\qbezier(10,30)(-2,25)(-9,20)

\qbezier(50,130)(-35,110)(10,90)
\qbezier(10,90)(50,75)(10,60)
\qbezier(10,60)(-30,45)(10,30)
\qbezier(10,30)(22,25)(29,20)


\qbezier(-30,145)(36,110)(10,90)
\qbezier(10,90)(-10,75)(10,60)
\qbezier(10,60)(30,45)(10,30)
\qbezier(10,30)(2,25)(-1,20)

\qbezier(50,145)(-16,110)(10,90)
\qbezier(10,90)(30,75)(10,60)
\qbezier(10,60)(-10,45)(10,30)
\qbezier(10,30)(18,25)(21,20)



\qbezier(10,-25)(0,-20)(-5,-15)
\qbezier(-30,-70)(45,-45)(10,-25)

\qbezier(10,-25)(20,-20)(25,-15)
\qbezier(50,-70)(-25,-45)(10,-25)


\qbezier(10,-25)(-2,-20)(-9,-15)
\qbezier(-30,-65)(55,-45)(10,-25)

\qbezier(10,-25)(22,-20)(29,-15)
\qbezier(50,-65)(-35,-45)(10,-25)


\qbezier(10,-25)(2,-20)(-1,-15)
\qbezier(-30,-78)(37,-45)(10,-25)

\qbezier(10,-25)(18,-20)(21,-15)
\qbezier(50,-75)(-15,-45)(10,-25)

\put(5,-40){$\cdots$}

\put(0,-70){$\cdot$}
\put(9,-70){$\cdot$}
\put(18,-70){$\cdot$}


\put(29,10){$\cdot$}
\put(29,2){$\cdot$}
\put(29,-6){$\cdot$}

\put(-12,10){$\cdot$}
\put(-12,2){$\cdot$}
\put(-12,-6){$\cdot$}


\put(-70,3){supersingular points}

\put(-40,15){\vector(1,2){18}} 
\put(-40,13){\vector(2,1){19}}
\put(-40,0){\vector(1,-1){20}}

\put(-40,24){\vector(1,3){24}}

\put(-21,63){$\cdot$}
\put(-21,60){$\cdot$}
\put(-21,57){$\cdot$}





\put(90,37){\vector(-1,0){39}}


\qbezier(105,115)(215,110)(305,117)

\put(322,116){$D_{i,s}$}

\put(195,120){$\cdot$}
\put(215,119){$\cdot$}
\put(230,120){$\cdot$}

\qbezier(103,92)(215,90)(305,97)

\put(323,95){$D_{i,s}$}

\put(195,96){$\cdot$}
\put(215,96){$\cdot$}
\put(230,97){$\cdot$}

\qbezier(104,62)(215,60)(305,68)

\put(195,66){$\cdot$}
\put(215,66){$\cdot$}
\put(230,67){$\cdot$}

\qbezier(103,32)(215,32)(305,38)

\put(195,36){$\cdot$}
\put(215,36){$\cdot$}
\put(230,37){$\cdot$}

\put(323,67){$\cdot$}
\put(322,47){$\cdot$}
\put(321,27){$\cdot$}
\put(322,7){$\cdot$}
\put(323,-16){$\cdot$}


\qbezier(103,-26)(215,-29)(305,-26)

\put(195,-22){$\cdot$}
\put(215,-23){$\cdot$}
\put(230,-23){$\cdot$}

\qbezier(104,-53)(215,-54)(305,-51)

\put(320,-53){$D_{i,s}$}
\put(195,-49){$\cdot$}
\put(215,-49){$\cdot$}
\put(230,-50){$\cdot$}


\qbezier(130,143)(135,0)(125,-80)

\put(122,13){$\cdot$}
\put(122,2){$\cdot$}
\put(121,-9){$\cdot$}

\qbezier(160,142)(165,0)(158,-82)

\put(152,13){$\cdot$}
\put(152,2){$\cdot$}
\put(151,-9){$\cdot$}

\qbezier(180,145)(184,0)(182,-79) 

\put(174,13){$\cdot$}
\put(174,2){$\cdot$}
\put(173,-9){$\cdot$}



\qbezier(245,145)(242,0)(252,-80) 

\put(249,13){$\cdot$}
\put(249,2){$\cdot$}
\put(250,-9){$\cdot$}

\qbezier(265,144)(261,0)(272,-80) 

\put(269,13){$\cdot$}
\put(269,2){$\cdot$}
\put(270,-9){$\cdot$}

\qbezier(285,145)(279,0)(292,-80)

\put(289,13){$\cdot$}
\put(289,2){$\cdot$}
\put(290,-9){$\cdot$}








\end{picture}
\end{center}
\caption{Katz-Mazur and stable fibers of ${X}_i$}
\label{FigureX_i}
\end{figure}

The situation at the supersingular points can be described as follows. 
Let $\mathrm{x}$ be a point of $X_{i,\F_p}$ whose underlying elliptic
curve $E_0$ is supersingular, and let $k$ be the residue field
of~$\mathrm{x}$. Then $\mathrm{x}$ is a triple $(E_0/k,\alpha_0,\phi_0)$
with $\alpha_0\in{\cal P}(E_0/k)$ and $\phi_0\in
[\Gamma(p)^{\mathrm{can}}](E_0/k)$; note that $\phi_0(1,0)$ and
$\phi_0(0,1)$ are both~$0$, as $E_0$ is supersingular. Let $R$ be the
completion of the local ring of $X_{i,\F_p}$ at~$\mathrm{x}$. 

By construction, $R$ is the universal formal deformation ring of
$(E_0,\alpha_0,\phi_0)$ to Artin local $k$-algebras. 
That is, restricting the universal triple over 
${\cal M} ({\cal P}, \Gamma (p)^{\mathrm{can}})_k$ to $R$ gives the
Cartesian diagram 
\[
\begin{tikzcd}
(E_{0,k},\alpha_0,\phi_0) \ar[r] \ar[d] \ar[dr, phantom, "\Box"] & (\E_R,\alpha^\univ,\phi^\univ) \ar[d]\\
\mathrm{Spec}(k) \ar[r] & \mathrm{Spec}(R)\,.
\end{tikzcd}
\]
This diagram has the property that 
for every Artin local $k$-algebra~$A$ with residue field~$k$,
every $(E/A,\alpha,\phi)$, and every Cartesian diagram 
\[
\begin{tikzcd}
(E_{0,k},\alpha_0,\phi_0) \ar[r] \ar[d] \ar[dr, phantom, "\Box"] & (E,\alpha,\phi) \ar[d]\\
\mathrm{Spec}(k) \ar[r] & \mathrm{Spec}(A)
\end{tikzcd}
\]
there are unique dashed maps
\[
\begin{tikzcd}
(E_{0,k},\alpha_0,\phi_0) \ar[r] \ar[d] \ar[rr, bend left] & (E,\alpha,\phi) \ar[d]
\ar[r, dashed] \ar[dr, phantom, "\Box"] &
(\E_R,\alpha^\univ,\phi^\univ) \ar[d] \\
\mathrm{Spec}(k) \ar[r] \ar[rr, bend right]& \mathrm{Spec}(A) \ar[r, dashed] & \mathrm{Spec}(R)
\end{tikzcd}
\]
that make the diagram commutative, and the right square Cartesian.

In order to get a useful description of~$R$, let $\E/k[[t]]$ be a
universal deformation of $E_0$ to Artin local $k$-algebras with
residue field~$k$ (see Section~\ref{0-1728} for some explicit
ones). As $\E_R$ is a deformation of $E_0$ over~$R$, we have a unique Cartesian diagram
\[
\begin{tikzcd}
E_{0,k} \ar[r] \ar[d] \ar[rr, bend left] & \E \ar[d]
\ar[r, dashed] \ar[dr, phantom, "\Box"] &
\E_R \ar[d] \\
\mathrm{Spec}(k) \ar[r] \ar[rr, bend right]& \mathrm{Spec}(k[[t]]) \ar[r, dashed] & \mathrm{Spec}(R)\,.
\end{tikzcd}
\]
As ${\cal P}$ is \'etale over $(\mathrm{Ell})_{\Z_p}$, $\alpha_0$ lifts
uniquely to every deformation of~$E_0$. Therefore, the connected
component of $({\cal P},[\Gamma(p)^{\mathrm{can}}])_{\E/k[[t]]}$
containing $\mathrm{x}$ is equal to the base change of 
$[\Gamma(p)^{\mathrm{can}}]_{\E/k[[t]]}$ via $k\to k'$, and hence 
\[
\mathrm{Spec}(R) = 
[\Gamma(p)^{\mathrm{can}}]_{\E/k[[t]]}\times_{\mathrm{Spec}(k)}\mathrm{Spec}(k')\,,
\]
that is, $\mathrm{Spec}(R)$ is the $k'[[t]]$-scheme representing all
$\Gamma(p)^{\mathrm{can}}$-structures on~$\E_{k'[[t]]}/k'[[t]]$. Being
this, $R$ is
a $k'[[t]]$-algebra, free of rank $\#\SL_2(\F_p)$ as $k'[[t]]$-module.

Let $Z$ be a parameter of the formal group of~$\E /k[[t]]$. Then, as
$E/k$ is supersingular, $\phi^\univ(1,0)$ and $\phi^\univ(0,1)$ in
$\E(R)$ are two points of that formal group. We write 
\[
x=Z(\phi(1,0))\in R,\quad y=Z(\phi (0,1))\in R 
\]
for their respective parameters. Katz and Mazur prove
in~\cite[\S5.4]{KM85} that $x$ and $y$ generate the maximal ideal
of~$R$, hence that $R$ is a quotient of the formal power series ring
$k'[[x,y]]$ whose $x$ and $y$ map to $x$, resp.~$y$ in~$R$. The fact
that $X_{i,\F_p}$ is the union of the ${\mathrm{Ig}}_{i,L}$ for $L$
running through $\PPP^1 (\F_p )$ means that the kernel of
$k'[[x,y]]\to R$ is generated by the product of equations of
the~${\mathrm{Ig}}_{i,L}$.  Now the condition that $\phi^\univ$
defines a point on~${\mathrm{Ig}}_{i,L}$ is
$$
\left\{
\begin{array}{ll}
\phi^\univ(1,0)=a\cdot \phi^\univ(0,1) & \quad\text{if $L=\F_p \cdot (1,-a)$} \\
\phi^\univ(0,1)=0 & \quad \text{if $L=\F_p \cdot (0,1)$}
\end{array}
\right.
$$ 
which translates on the formal group, for $\tilde{a}\in\Z_p$ any lift
of $a\in\F_p$, as
$$
\left\{
\begin{array}{ll}
x=[\tilde{a}](y) =ay +{\mathrm{higher\ degree\ terms\ in\ }} y  
& \hspace{0.1cm} {\mathrm{if\ }} L=\F_p \cdot (1,-a) \\
y=0 & \hspace{0.1cm} {\mathrm{if\ }} L=\F_p \cdot (0,1).
\end{array}
\right.
$$ 
The equation $f$ in $k'[[x,y]]$ mod $(x,y)^{p+2}$ is therefore
$y\prod_{a\in \F_p} (x-ay)=x^p y -xy^p$. The regularity of $X_i$
at~$\mathrm{x}$, plus~\cite[Thm.~13.8.4]{KM85} give the following.
\begin{theo}[Katz-Mazur]
\label{KatzMazur2}
The complete local ring of the arithmetic surface $X_i$ at 
a supersingular point~$\mathrm{x}$ is isomorphic to
$$
W(k')[\zeta_p ][[x,y]]/(x^p y -xy^p +g+(1-\zeta_p )f_1)\,,
$$ 
with $k'$ the residue field at~$\mathrm{x}$, $W(k' )$ its ring of Witt vectors, $g$ belongs to the ideal
$(x,y)^{p+2}$ and $f_1$ is a unit of $W(k')[\zeta_p][[x,y]]$. 
\end{theo}

\subsection{The stable model}\label{stable_model_Gamma(p)}

We can now describe how to compute the (semi)stable model of
$\overline{\cal M}({\cal P},[\Gamma(p)])$ over ``the'' completely
ramified degree-$(p^2 -1)$ extension of~$\Z_p^{\mathrm{ur}}$.  (Here and in all
what follows, $\Z_p^{\mathrm{ur}}$ denotes the ring of integers of the maximal unramified extension of
$\Q_p^{\mathrm{ur}}$ of $\Q_p$.)

First we recall a general tool for explicitly computing semistable
models of curves in tame situations, starting from a regular
model. Let $S$ be the spectrum of some discrete valuation ring, whose
generic and closed point we denote by $\eta$ and $s$ respectively. Let
${\cal C}\to S$ be a curve, that is, a $S$-scheme purely of relative
dimension~$1$. Assume ${\cal C}$ is proper and flat over $S$, that
${\cal C}$ is regular, and ${\cal C}_\eta :=C \to S_\eta$ is
smooth. By \cite[Thm~2.26]{Liu02}, (and \cite[Rem.2.27]{Liu02}), after
sufficiently many blow-ups in closed singular points of ${\cal C}$ we
can assume ${\cal C}_s$ is a Cartier divisor on ${\cal C}$ with normal
crossings. Write $n$ for the least common multiple of the
multiplicities of irreducible components of ${\cal C}$ and set
$T:=S[\pi_0^{1/n}]$, for $\pi_0$ some uniformizer on~$S$. Let
$\widetilde{{\cal C}_T}$ be the normalization of the base change
${\cal C}\times_S T \to T$.  Then, assuming $n$ is prime to~$p$, one
knows that $\widetilde{{\cal C}_T}_{/T}$ is a semistable curve: the
only singularities of the geometric special fibre are ordinary double
points, that is, with complete local ring isomorphic to that of the
union of the $2$ coordinate axes in the affine plane, at the
origin. In the case of complex surfaces, this knowledge comes from the
resolution of Hirzebruch--Jung singularities (\cite{Hirzebruch}
and~\cite{Jung}), see~\cite[III.\S5]{BHPvdV} and the historical
remarks at the end of~\cite[III]{BHPvdV}. See \cite[\S2.1]{CES}
for the case we use in this article. For the case where one starts
with a curve ${\cal C}\to S$ with $\cal C$ not necessarily regular, see
\cite[8]{Liu02} and \cite[\S2.1]{CES} for resolution of singularities.

\begin{rema}
\label{higerp^n}
{\rm 
\begin{itemize}
\item From our semistable model it is not hard to obtain a stable one
  via appropriate contractions. 
\item In the case of modular curves, the hypothesis that $n$ be prime
  to $p$ is typically {\em not} satisfied when the level is divisible by
  $p^2$. For those 
more difficult cases rigid analytic methods are more succesful, as
shown in the work of Weinstein (\cite{Wein16}; see also references in
the Introduction of loc. cit.). 
\end{itemize}
}
\end{rema}

We apply the above to compute a semistable model for $\overline{\cal
  M}({\cal P},[\Gamma(p)])$. Actually, for $p$ not too small, that
model will happen to be even {\it stable}.  Starting from the regular
curve $\overline{\cal M}({\cal P},[\Gamma(p)])^{\sim}_{\Z_p[\zeta_p]}$
over~$\Z_p[\zeta_p]$, equal to $\coprod_i X_i$ by~(\ref{X_i}), we
sum-up the algorithm we follow:
\begin{enumerate}
\item blow-up singular points in the closed fiber until having normal crossings; 
\item provided the l.c.m.~$n$ of multiplicities of components is prime
  to $p$ (which will be the case for us), base-change to ``the''
  purely ramified-at-$p$ extension of $\Q_p$ of degree $n$ and
\item normalize; we denote the result by~$\overline{\cal
  M}({\cal P},[\Gamma(p)])^{\mathrm{st}}$.    
\end{enumerate} 

It is clear from that construction and Theorem~\ref{KatzMazur1} that
the special fiber of our semistable model $\overline{\cal M} ({\cal
  P},[\Gamma (p)]) ^{\mathrm{st}}$ over $\Z [(1-\zeta_p )^{1/p+1}]$
will have two types of irreducible components: the ``vertical ones'',
obtain by simple base change from the components of
Katz-Mazur model, and the ``horizontal ones'', which contract to
supersingular points in that model. The former vertical components, which
are copies of the $\overline{\cal M} ({\cal P},
[{\mathrm{ExIg}} (p,1 )])$, will be called {\it Igusa parts}. The
latter horizontal ones will be referred to as {\it Drinfeld  components} and computed in next section.

\subsubsection{Drinfeld components}
\label{Drinfeld}

We know from Theorem~\ref{KatzMazur2} that the complete local ring
of~$X_i$ at some singular point~$s$ is $W(k)[\zeta_p][[x,y]] /(f)$,
for $k$ the residue field of~$s$, and $f=x^p y -xy^p +g+(1-\zeta_p
)f_1$, with $g$ in the ideal $(x,y)^{p+2}$ and $f_1$ a unit of
$W(k)[\zeta_p][[x,y]]$. The completion along the exceptional divisor
of the blow up of~$X_i$ in~$s$ is therefore covered by two affine open
${\mathrm{Spf}} (A_1 )$ and ${\mathrm{Spf}} (A_2 )$, with $A_1
=W(k)[\zeta_p ][v][[x]] /f(x,vx)$ and $A_2 =W(k)[\zeta_p ][u][[y]] /f(uy,y)$
(see~\cite{Edix89}, 1.3.1). So here
$$
A_1 = W(k)[\zeta_p ][v][[x]] /(x^{p+1} v(1-v^{p-1}) +g+(1-\zeta_p )f_1 (x,xv))
$$
which shows that the exceptional divisor in ${\mathrm{Spf}} (A_1 )$
has multiplicity $p+1$, and same with~$A_2$. So we extend the base
ring $W(k)[\zeta_p]$ to $W(k)[\pi]$ with
\begin{eqnarray}
\label{pi} 
\pi :=(1-\zeta_p )^{1/(p+1)}
\end{eqnarray}
so that, writing $g(x,vx) =x^{p+2}h$,
$$
A_1 \otimes_{W(k)[\zeta_p  ]} W(k)[\pi] = 
W(k)[\pi ][v][[x]] /(x^{p+1} v(1-v^{p-1}) +x^{p+2}h +\pi^{p+1}f_1 (x,xv)) 
$$  
and, to normalize it, we blow up at $(x,\pi )$. This means we set $\pi =xw$ and 
$$
A_1 \tilde{\otimes}_{W(k)[\zeta_p  ]} W(k)[\pi] = 
W(k)[\pi ][v][[x]] [w] /(\pi -xw ,(v-v^{p}) +xh +w^{p+1}f_1 (x,xv)) .
$$
The corresponding affine part of the exceptional divisor $D_{i,s}$
above $s$ is given by $x=0$, so that $D_{i,s}$ has an affine model with equation 
\begin{eqnarray}
aw^{p+1}=v^p -v
\end{eqnarray}
for $a\in k^*$ the image of~$f_1 (0,0)$. One could possibly determine that $a$ but we will content
ourselves in that paper with geometric models so we will henceforth assume $a=-1$. Putting $\alpha =1/w$
and $\beta =v/w$ gives the other model
\begin{eqnarray}
\label{modelDis}
\alpha^{p} \beta -\alpha \beta^p =1 .
\end{eqnarray}

 Note the singularities of our model have thickness $1$.

Keeping track of our parameters, we register that
\begin{eqnarray}
\label{lesparametres}
\alpha =w^{-1} =\pi^{-1} x \hspace{0.5cm} {\mathrm{and}} \hspace{0.5cm} \beta =w^{-1} v=\pi^{-1} vx =\pi^{-1} y.
\end{eqnarray}

\begin{rema}
\label{supersinguliers}
{\rm The above Drinfed components are supersingular (i.e. have supersingular jacobians), which means that their quotients
showing-up in the models below are, too. Indeed, from~(\ref{modelDis})  we know they have geometric projective equation $X^p Y -XY^p =Z^{p+1}$
(a so-called ``Hermitian equation''). Hurwitz formula shows their genus is $g=\frac{1}{2} p(p-1)$. Considering the 
form $H_s$ given by the equation $X^p Y -XY^p =aZ^{p+1}$, for $a\in \F_{p^2}$ some non-trivial square root of an 
element in $\F_p$, one checks that the number $\# H_s (\F_{p^2} )$ of points of $H_s$ with values 
in  $\F_{p^2}$ is $(p^3 +1)$. That therefore means that $H_s$ is {\it maximal over $\F_{p^2}$}: 
$$
\# H_s (\F_{p^2} )=1+p^2 +2pg
$$
is the maximum allowed by Weil's bound. Now if $(a_i )_{1\le i\le 2g}$ is the set of eigenvalues of the Frobenius 
endomorphism of $\mathrm{Jac} (H_s )$ over $\F_p$, we have $\# H_s (\F_{p^2} ) =1+p^2 -\sum_{i=1}^{2g} a_i^2$, 
and Riemann's hypothesis $\vert a_i\vert \leq \sqrt{p}$ implies that $a_i^2 =-p$, so that Frobenius has 
characteristic polynomial $(X^2 +p)^g$. The latter is the characteristic polynomial of Frobenius on $E^g$, for $E$ a 
supersingular elliptic curve over $\F_p$.
}
\end{rema}

\subsubsection{Points with exceptional automorphisms}
\label{0-1728}

In order to compute stable models for level structures defining
non-rigid moduli problems, that is, to compute stable models for
coarse moduli spaces, we shall consider quotients of the above stable
models $\overline{{\cal M}}({\cal P},[\Gamma (p)])^{\mathrm{st}}$ by
relevant subgroups $H\subseteq \GL_2 (\F_p )$, such as
$H=\Gamma_{\mathrm{s}}(p)$, $\Gamma_{\mathrm{ns}}(p)$,
$\Gamma_{\mathrm{s}}^+ (p)$ or $\Gamma_{\mathrm{ns}}^+(p)$, that is, the split
or non-split Cartan subgroup, or their respective normalizers. Then to
get rid of the rigidifying level structure ${\cal P}$, we shall assume
it is representable, finite \'etale over $(\mathrm{Ell})_{\Z_p}$, and
Galois of group $G$; finally we take the quotient of our
$\overline{{\cal M}}({\cal P},[H])^{\mathrm{st}}$ by the action
of~$G$. To describe the local situation above singular points of
$\overline{{\cal M}}({\cal P},
H^{\zeta_p\mathrm{-can}})_{\Z_p[\zeta_p]}$ with extra automorphisms,
we however need to describe the action of those automorphisms on the
relevant deformation rings.

So let $E_0$ be a supersingular elliptic curve over $k:=\F_{p^2}$,
such that ${\mathrm{Aut}}_k (E_0 )$ is cyclic of order $4$ ($j=1728$)
or $6$ ($j=0$). Let $\mathrm{x}$ be a point of ${\cal M} ({\cal P},
\Gamma (p)^{\mathrm{can}})_k$ whose underlying elliptic curve
is~$E_0$, and let $k'$ be its residue field. Then $\mathrm{x}$ is a
triple $(E_{0,k'} ,\alpha_0 ,\phi_0 )$ with $\alpha_0$ in ${\cal
  P}(E_{0,k'}/k')$ and $\phi_0$ the unique (trivial) element of
$[\Gamma(p)^{\mathrm{can}}](E_{0,k'}/k')$. Let $R$ be the completion
of the local ring of ${\cal M} ({\cal P}, \Gamma (p)
^{\mathrm{can}})_k$ at~$\mathrm{x}$. 
In order to get a useful description of~$R$, we first give
a universal deformation of $E_0$ to Artin local $k$-algebras with
residue field~$k$.

If $j=1728$, one can check (cf.~\cite{Edix89}, 1.3.2) that the
elliptic curve $\E$ over $k[[t]]$ given by the Weierstrass equation
\begin{equation}
\label{Weierstrassj=1728}
Y^2 =X^3 -X+t
\end{equation} 
is universal. (Indeed, it is well-known that one can choose for $E_0$
an equation of shape $Y^2 =X^3 -X$. Any deformation of $E_0$ to an
Artinian local $k$-algebra $A$ with residue field $k$ and maximal
ideal $m$ can then be given an equation $Y^2 =X^3 +a X +b$, with $a+1$
and $b$ in $m$. (Recall that $p>3$.) Now one can write $a=-c^4$ for $c\in A$ congruent
to $1$ mod $m$. Replacing the variables $X$ and $Y$ by $c^{-2} X$ and
$c^{-3} Y$ respectively gives the desired model for $\E$.)  
The action of a generator $i$ of
$\mu_4 (k) = {\mathrm{Aut}}_k(E_0 )$ (via action on tangent space
at~$0$) is given by:
\[
[i]\colon X\mapsto -X, \quad Y\mapsto iY,\quad t\mapsto -t.
\] 

In the case $j=0$ one similarly sees that a model for $\E$ over
$k[[t]]$ is given by the Weierstrass equation 
\begin{equation}
\label{Weierstrassj=0}
Y^2 =X^3 +tX-1
\end{equation} 
with automorphism action given by:
\[
[\zeta ]\colon X\mapsto \zeta^{-2} X, \quad Y\mapsto -Y, \quad t\mapsto
\zeta^2 t
\]
for $\zeta$ some generator 
of $\mu_6 (k)={\mathrm{Aut}}_k (E_0 )$ (again, identification via the
action on the tangent space at~$0$). 

As ${\cal P}$ is \'etale over $(\mathrm{Ell})_{\Z_p}$, $\alpha_0$ lifts
uniquely to every deformation of~$E_0$. Therefore, 
\[
\mathrm{Spec}(R) = 
[\Gamma(p)^{\mathrm{can}}]_{\E/k[[t]]}\times_{\mathrm{Spec}(k)}\mathrm{Spec}(k')\,.
\]
To arrive at the description  of $R$ by  Katz-Mazur in
Theorem~\ref{KatzMazur2}, we choose $X/Y$ as parameter $Z$ of the
formal group of $\E /k[[t]]$, with $X$ and $Y$ the functions of the
Weierstrass model above. Our description of the action of $[i]$ and
$[\zeta]$ shows that
$[i] \colon Z\mapsto iZ$, and $[\zeta] \colon Z\mapsto \zeta Z$. Therefore as $\phi (1,0)$ and $\phi (0,1)$ constitute 
our Drinfeld basis of~$\E [p]$, we have $x=Z(\phi (1,0))$ and 
$y=Z(\phi (0,1)) \in R$ for their respective parameters as in
Theorem~\ref{KatzMazur2} so that $[i]$ maps them to $ix$ and $iy$
respectively, and $[\zeta]$ maps them to $\zeta x$ and~$\zeta y$. 
It means the parameters $\alpha$ and $\beta$ of
equations~(\ref{lesparametres}) are mapped to $i\alpha$ and $i\beta$
respectively, and similarly to $\zeta \alpha$ and~$\zeta \beta$. (One immediately
checks that equation~(\ref{modelDis}) is preserved because
$p+1$ is divisible by the order of the automorphism.)

\subsubsection{The action of $\GL_2 (\F_p )$}
\label{GL2}

The action of $\GL_2 (\F_p )$ on ${\cal M}({\cal P},\Gamma(p))$ from the
right has the obvious modular interpretation:
$$
r(g) \colon (E/S,\alpha,\phi )\mapsto (E/S,\alpha,\phi\circ g).
$$ 
By construction, that extends uniquely to an action on our semistable
model $\overline{{\cal M}}({\cal P},\Gamma (p))^{\mathrm{st}}$, and we
want to describe this on the special fiber.  As $e_p (\phi \circ
g(1,0), \phi \circ g(0,1)) =e_p (\phi (1,0), \phi (0,1))^{\det (g)}$,
the action of $g$ on $\coprod_i X_i \otimes_{\Z [\zeta_p ]}
\overline{\F}_p$ is
$$
\left(  (E/S/ \overline{\F}_p,\alpha,\phi ),i \right) \mapsto \left(
  (E/S/ \overline{\F}_p,\alpha ,\phi \circ g ),i\det (g) \right) .  
$$
The action of $\GL_2 (\F_p )$ on the ${\mathrm{Ig}}_{i,P}$ goes
therefore as follows. Each $g$ induces an isomorphism  
$$
r(g)\colon {\mathrm{Ig}}_{i,P} \stackrel{\sim}{\longrightarrow} {\mathrm{Ig}}_{i\det (g),g^{-1} P} 
$$
so that the stabilizer of ${\mathrm{Ig}}_{i,P}$ is the Borel subgroup $B_P$ of ${\mathrm{SL}}_2 (\F_p )$ that fixes the line $P$. 
%
%
As for the Drinfeld components, $g$ induces an isomorphism
$$
r(g) \colon D_{i,s} \stackrel{\sim}{\longrightarrow} D_{i\det (g), s} 
$$
and the stabilizer of $D_{i,s}$ is ${\mathrm{SL}}_2 (\F_p )$. Recalling the notation we have introduced before 
Theorem~\ref{KatzMazur2} we denote by $Z$ a parameter of the formal group of the universal deformation 
$\E /\overline{\F}_p [[t]]$, so that our universal $p$-torsion basis have parameters $x=Z(\phi (1,0))$ and $y=Z(\phi (0,1))$.
Writing $g=\left( \begin{array}{cc}
a & b\\
c & d
\end{array}
\right)$ in ${\mathrm{SL}}_2 (\F_p )$, we see that $g$ acts from the left on $W(k' ) [\zeta_p ] [[x,y]]/(f)$ by:
\begin{eqnarray}
r(g)^{\#} x & = & r(g)^{\#} Z(\phi (1,0))= Z(\phi \circ g(1,0)) =Z(\phi (a,c))  \nonumber \\
& \equiv  &  aZ(\phi (1,0)) + bZ(\phi (0,1)) \mod (x,y)^2 \equiv  ax+cy \mod    (x,y)^2  \\
r(g)^{\#} y & \equiv  & bx+dy \mod   (x,y)^2  .
\end{eqnarray}  
It therefore follows from~(\ref{lesparametres}) that $g$ acts on our model~(\ref{modelDis}) by 
\begin{eqnarray}
\label{gonalpha}
  r(g)^{\#} \alpha = a\alpha +c\beta \hspace{0.5cm} {\mathrm{and}}
  \hspace{0.5cm} r(g)^{\#} \beta = b\alpha +d\beta \ ;
\end{eqnarray}
one readily checks that equations~(\ref{modelDis}) are preserved by the action of ${\mathrm{SL}}_2 (\F_p )$.

\subsubsection{Galois action}
\label{Inertia}

Let $G_\Q$ be the absolute Galois group of $\Q$, and $G_p$ its
decomposition group at a maximal ideal of $\overline{\Z}$ over~$p$,
which we identify with the absolute Galois group $G_{\Q_p}$ of~$\Q_p$.
If $\Q_p^{\mathrm{ur}}$ and $\Q_p^{\mathrm{t}}$ are the usual
notations for the maximal unramified and tame extension of $\Q_p$
respectively, the sequence of inclusions $\Q_p \subset
\Q_p^{\mathrm{ur}} \subset \Q_p^{\mathrm{t}} \subset \overline{\Q}_p$
induces the sequence of Galois subgroups
$$
I_p \triangleleft I \triangleleft G_p \triangleleft G_\Q
$$
where correspondingly $I$ is the inertia subgroup, and $I_p$ its wild inertia subgroup. The tame inertia group $I_t :=I/I_p \simeq {\mathrm{Gal}} (\Q_p^{\mathrm{t}} 
/\Q_p^{\mathrm{ur}})$ can be identified with $\varprojlim_{p \nmid  n} \mu_n (\overline{\F}_p )$ (where $\mu_n$ stands for the $n^{\mathrm{th}}$-roots
of unity) by 
\begin{eqnarray}
\label{sigma}
\sigma \mapsto (\sigma (p^{1/n} ) /p^{1/n} ))_{p\, \nmid  \, n}
\end{eqnarray}
(so that the transition morphisms $\mu_{nm} \to \mu_n$ are given by $\zeta_{nm} \mapsto \zeta_{nm}^m$), and that is still isomorphic 
to  $\varprojlim \F_{p^n}^*$ (in which transition morphisms are now given by the norm): this is Serre's theory of 
``caract\`eres fondamentaux'', cf.~\cite{Se72}, paragraph~1.3. 
Any $\sigma$ in $G_\Q$ induces an automorphism
\[
\gamma (\sigma ) :={\mathrm{id}}\times {\mathrm{Spec}} (\sigma )\quad\text{of}\quad 
\overline{\cal M} ({\cal P},\Gamma (p)) \times_\Q {\mathrm{Spec}}
(\overline{\Q} )
\]
with $\pi =(1-\zeta_p )^{\frac{1}{p+1}}$ as in~(\ref{pi}). The above fiber product is also 
$\overline{\cal M} ({\cal P},\Gamma (p))_{\Q (\pi)} \times_{\mathrm{Spec} (\Q (\pi ) )} 
\mathrm{Spec} (\overline{\Q} )$
and $\gamma (\sigma )$ extends uniquely to an automorphism of
$\overline{\cal M} ({\cal P},\Gamma (p))^{\mathrm{st}}_{\Z [\pi ]} \times_{{\mathrm{Spec}} 
(\Z [\pi ])} \mathrm{Spec} (\overline{\Z} )$
that we still denote by $\gamma (\sigma )$. It follows that any $\sigma$ in $G_{\Q_p}$ induces an automorphism of the special fiber
$\overline{\cal M} ({\cal P}, \Gamma (p))^{\mathrm{st}}_p
:=\overline{\cal M} ({\cal P},\Gamma (p))^{\mathrm{st}}_{\Z [\pi ]} \times_{{\mathrm{Spec}} (\Z  [\pi ])} \overline{\F}_p$, 
and if $\sigma$ actually belongs to $I$ then $\gamma (\sigma )$ is an $\overline{\F}_p$-automorphism.  

The extension $\Z \to \Z [\pi ]$ having degree $p^2 -1$ and being
totally ramified at $p$, the inertia action just defined factorizes
through an antihomomorphism $\overline{\gamma} \colon \F_{p^2}^* \to
{\mathrm{Aut}}_{\overline{\F}_p} (\overline{\cal M} ({\cal P},\Gamma
(p))^{\mathrm{st}}_p )$.  The action of $I$ on $\mu_p (\F_p )$
factorizes through $I\to \F_{p^2}^* \stackrel{N}{\longrightarrow}
\F_p^*$, for $N$ the norm map $x\mapsto x^{p+1}$.  This means that the
action of inertia on the left-hand side $\coprod_{i\in \F_p^*} X_i$
of~(\ref{X_i}) has the modular interpretation:
$$
(E/S \stackrel{f}{\longrightarrow} {\mathrm{Spec}} (\Z [\zeta_p ]) ,\phi )_i \stackrel{\overline{\gamma}(u)}{\longmapsto} (E/S \stackrel{f'}{\longrightarrow} 
{\mathrm{Spec}} (\Z [\zeta_p ])  ,\phi )_{iu^{-p-1}}
$$  
for $f' ={\mathrm{Spec}} (u^{p+1}) \circ f$. It follows that $u \in \F^*_{p^2}$ induces isomorphisms:
\begin{eqnarray}
\label{gammabar}
\overline{\gamma} (u)\colon
\left\{
\begin{array}{rcl}
{\mathrm{Ig}}_{i,P} & \stackrel{\sim}{\longrightarrow} & {\mathrm{Ig}}_{iu^{-p-1},P} \\
D_{i,P} & \stackrel{\sim}{\longrightarrow} & D_{iu^{-p-1},P}  \hspace{0.1cm} .
\end{array}
\right. 
\end{eqnarray}
The stabilizer in $\F^*_{p^2}$ of both types of components is the kernel $\mu_{p+1} (\F_{p^2} ) = (\F_{p^2}^*)^{p-1}$ of the norm map. Because the ${\mathrm{Ig}}_{i,P}$
are already components of the special fiber of some $\Z [\zeta_p ]$-scheme, $\mu_{p+1} (\F_{p^2} )$ acts trivially on each of them. As for the $D_{i,s}$,
we see from~(\ref{sigma}) that $\overline{\gamma} (u) (\pi )=u\pi \mod \pi^2$, so that~(\ref{lesparametres}) implies that $u\in \mu_{p+1} (\F_{p^2} )$ induces the
automorphism
\begin{eqnarray}
\label{gammabarDis}
\overline{\gamma} (u)^\# \colon \alpha \mapsto u^{-1} \alpha \hspace{0.5cm} {\mathrm{and}} \hspace{0.5cm}  \beta \mapsto u^{-1} \beta 
\end{eqnarray}
on the model~(\ref{modelDis}) of $D_{i,s}$.


\section{Non-split Cartan structures}
\label{NonSplitCartan}

\subsection{Stable model for $\overline{{\cal M}} ({\cal P}, \Gamma_{\mathrm{ns}} (p))$}

We compute the stable model for modular curves associated with a non-split Cartan group $\Gamma_{\mathrm{ns}} (p)\subseteq \GL_2 (\F_p )$
(but {\em not} its normalizer), endowed with some additional level structure ${\cal P}$.
 
\begin{theo} 
\label{NS} 
Let $p>3$ be a prime, and let $[\Gamma_{\mathrm{ns}} (p)]$ be the moduli problem over $\Z [1/p]$ associated with $\Gamma_{\mathrm{ns}} (p)$.
Let ${\cal P}$ be a representable moduli problem, which is finite \'etale over $({\mathrm{Ell}})_{/\Z_p}$  
(take for instance ${\cal P} =[\Gamma (N) ]$  for some $N\ge 3$ not divisible by $p$). 
Let $\overline{{\cal M}} ({\cal P}, \Gamma_{\mathrm{ns}} (p)) =\overline{{\cal M}} ({\cal P}, \Gamma (p))/ \Gamma_{\mathrm{ns}} (p)$ be the associated 
compactified fine moduli space. 
Let $W$ be a totally ramified extension of $\Z_{p}^{\mathrm{ur}}$ of degree $(p^2 -1)/2$ (for instance $W:=\Z_{p}^{\mathrm{ur}} [(1-\zeta_p )^{2/(p+1)}] $. 
Recall that $\Z_{p}^{\mathrm{ur}}$ denotes the ring of integers of the maximal unramified extension $\Q_{p}^{\mathrm{ur}}$ of $\Q_p$).

   Then $\overline{{\cal M}} ({\cal P}, \Gamma_{\mathrm{ns}} (p))$ has a semistable model over $W$ whose special fiber is made of two vertical Igusa parts,
which are linked by horizontal Drinfeld components above each supersingular point of $\overline{{\cal M}} ({\cal P})$ via the projection
$\overline{{\cal M}} ({\cal P}, \Gamma_{\mathrm{ns}} (p)) \to \overline{{\cal M}} ({\cal P})$.

\medskip 
 
 Both vertical parts, call them ${\mathrm{Ig}} (p ,{\cal P})_1$ and ${\mathrm{Ig}} (p, {\cal P})_d$ (for $d\in \F_p^*$ a non-square),
are isomorphic to the enhanced Igusa curve $\overline{{\cal M}} ({\cal P}, {\mathrm{Ig}} (p)/\{ \pm 1\} )_{\overline{\F}_p} $.

\medskip

  If ${\cal S}_{\cal P}$ is the number of supersingular points of $\overline{{\cal M}} ({\cal P}) (\overline{\F}_p)$,
the ${\cal S}_{\cal P}$ horizontal (Drinfeld) components are all copies of some hyperelliptic smooth curve $D$ for which an affine model is given by 
\begin{eqnarray}
\label{DrinfeldXns0}
U^2 =V^{p+1} +A_{\mathrm{ns}} 
\end{eqnarray}
for some $A_{\mathrm{ns}} $ in $\overline{\F}_p^*$.

\medskip

With $\pi_0$ a uniformizer of $W$ (e.g.
$\pi_0 =(1-\zeta_p )^{2/(p+1)}$), the completed local rings of the
singular points in the special fiber are isomorphic to
$W[[x,y]]/(xy-\pi_0 )$.
  
%
%
\end{theo}

\begin{rema}
\label{bisrepetitas}
{\rm Recall, as in Remark~\ref{semistableouquoi}, that we would have liked to call ``vertical components'' our ``vertical parts''
${\mathrm{Ig}} (p ,{\cal P})_1$ and ${\mathrm{Ig}} (p, {\cal P})_d$ above, but were formally prevented from doing so because ${\overline{\cal M}} ({\cal P})$
may not be irreducible itself.

    Finally, the constant terms $A_{\mathrm{ns}}$ in equations~(\ref{DrinfeldXns0}) could obviously been taken as $1$, as we here are only interested
in geometric models; the same holds for similar terms in the forthcoming parallel statements about split Cartan curves, etc. We leave that 
presentation as a reminder that a more precise determination could possibly be computed some day.} 
\end{rema}

For a picture of the curve we refer to Figure~\ref{figureNS}: it actually represents the coarse quotient $X_{\mathrm{ns}} (p)$, but that does not
affect the general shape.

\begin{proof}
Let $W'$ be the ramified quadratic extension $\Z_{p}^{\mathrm{ur}}[(1-\zeta_p )^{1/(p+1)}]$ of~$W$. 
One starts with the semistable model of $\overline{{\cal M}} ({\cal P}, \Gamma (p))$ of over $W'$ as described in Section~\ref{X(p)}, and takes the
quotient by the non-split torus $\Gamma_{\mathrm{ns}}(p) $ in $\GL_2 (\F_p)$ fixed above. This quotient is a semistable model of $\overline{{\cal M}} ({\cal
  P}, \Gamma_{\mathrm{ns}} (p))$  over~$W'$, with an action by $\F^*_{p^2}$ as described in Section~\ref{Inertia} (note that the $\GL_2 (\F_p)$-action
of Section~\ref{GL2} and the Galois action in Section~\ref{Inertia} commute with each other). The Galois group of $W'$ over $W$ is the subgroup
$\{\pm1\}$ of~$\F^*_{p^2}$. We will check that $\{\pm1\}$ acts trivially on the special fibre of our semistable model over~$W'$. Then the quotient
of that model over~$W'$ by $\{\pm1\}$ is the promised model over~$W$, and its pullback to $W'$ is our semistable model over~$W'$.

Recall (Section~\ref{stable_model_Gamma(p)}) that the vertical Igusa parts ${\mathrm{Ig}}_{i,P}$ are
indexed by $\F_p^* \times \PPP^1 (\F_p )$, and the action of $\GL_2 (\F_p )$ on the latter set is given by
$$
(i,(a\colon b)) \stackrel{g}{\mapsto} (i\det(g) ,g^{-1} {\cdot} (a\colon b)).
$$
If $D\simeq \F_p^*$ denotes the subgroup of scalar matrices, the action of $\Gamma_{\mathrm{ns}} (p) $ on $\PPP^1 (\F_p )$ factorizes via the quotient
$\Gamma_{\mathrm{ns}} (p) /D \simeq \Z/(p+1)\Z$ and that action is free and transitive. (Indeed the orbits on $\PPP^1 (\F_{p^2} )$, say, have size
$p+1$ or $1$, and $\PPP^1 (\F_p )$ is preserved.) One can therefore choose as representatives for the cosets
$(\F_p^* \times \PPP^1 (\F_p ) )/\Gamma_{\mathrm{ns}} (p) $ the two elements $(1,(1\colon 0))$ and $(d,(1\colon 0))$ for $d$ some non-square in
$\F_p^*$. Each Igusa component ${\mathrm{Ig}}_{i,P}$ has stabilizer $\pm 1$ in $\Gamma_{\mathrm{ns}} (p) $, so the two vertical parts are
isomorphic to $\overline{{\cal M}} ({\cal P}, {\mathrm{Ig}} (p)/\{ \pm 1\} )_{\overline{\F}_p }$. And indeed, the Galois group of $W'$ over $W$ acts
trivially on each of these two parts because, as noted at the end of Section~\ref{Inertia}, the group $\mu_{p+1}(\F_{p^2})$ acts trivially on
the~${\mathrm{Ig}}_{i,P}$.  That is for the first part of the Theorem.

Let us deal with the Drinfeld components. Recall that an equation for them in the bad fiber of $\overline{{\cal M}} ({\cal P}, \Gamma (p) )$  is given by
\begin{eqnarray}
\label{4.3.2}
-a =\alpha \beta^p -   \alpha^p \beta
\end{eqnarray}
for some $a$ in $\overline{\F}_p^*$ (cf. Section~\ref{X(p)}, (\ref{modelDis})).  Equations~(\ref{gonalpha}) and~(\ref{gammabarDis}) show that the
elements denoted~$-1$ in $\Gamma_{\mathrm{ns}} (p)$ and in $\F_{p^2}^*$ both act as $\alpha\mapsto -\alpha$ and $\beta\mapsto -\beta$. So, indeed, the
Galois group of $W'$ over $W$ acts trivially on the quotient by~$\Gamma_{\mathrm{ns}} (p)$.
To be completely explicit we choose some multiplicative generator $\kappa$ of $\F_{p^2}^*$ and pick $\lambda :=\kappa^{p-1}$ as a generator of the cyclic 
subgroup of elements with norm $1$ within $\F_{p^2}^*$. That subgroup is precisely the stabilizer in $\Gamma_{\mathrm{ns}} (p) $ of any Drinfeld component~$D_{i,s}$ 
by Section~\ref{GL2}. Set now $P_1 :=\lambda e_1 +\lambda^p e_2$, $P_2 :=\lambda^p e_1 +{\lambda} e_2$, for $(e_1, e_2 )$ the canonical basis, say, 
of $\F_{p}^2$. We can choose $\Gamma_{\mathrm{ns}} (p) $ so that it acts diagonally on this basis $(P_1 ,P_2 )$, that is, $\Gamma_{\mathrm{ns}} (p)$
can be written as  $\{ \left(
\begin{array}{cc}
a^p & 0\\
0 & a
\end{array}
\right) ,
a\in \F_{p^2}^* \}$ with respect to $(P_1 ,P_2 )$.
We perform the change of coordinates
\begin{eqnarray}
\label{coordinates}
\tilde{\alpha} :=\lambda \alpha + \lambda^p \beta \hspace{0.4cm} \mathrm{and} \hspace{0.4cm} \tilde{\beta} :=\lambda^p \alpha +\lambda \beta .
\end{eqnarray}
Then if $N:=(\lambda^{-2} -\lambda^2 )$ one has
\begin{eqnarray}
\label{inversecoordinates}
{\alpha} =\frac{1}{N} (-\lambda \tilde{\alpha} + \lambda^p \tilde{\beta} ) \hspace{0.4cm} \mathrm{and} \hspace{0.4cm} \beta =\frac{1}{N} ( \lambda^p \tilde{\alpha} -\lambda \tilde{\beta} )
\end{eqnarray}
from which equation~(\ref{4.3.2}) becomes
\begin{eqnarray}
\label{4.3.3}
a N={\tilde{\alpha}}^{p+1} -  {\tilde{\beta}}^{p+1}  .
\end{eqnarray}
Now
%
coordinates for the quotient curve $D_s :=D_{i,s} /(\Gamma_{\mathrm{ns}} (p)  \cap \SL_2 (\F_p ))$ are 
\begin{eqnarray}
\label{u1v1}
\left\{
\begin{array}{l}
u_1 :=\tilde{\alpha}^{p+1} \\
v_1 :=\tilde{\alpha} \tilde{\beta}
\end{array}
\right.
\end{eqnarray}
(they are indeed stable under the action of $\Gamma_{\mathrm{ns}} (p)  \cap \SL_2 (\F_p )$, and the corresponding morphism of curves has 
due degree $p+1$) so an equation for $D_s$ is
\begin{eqnarray}
\label{4.3.4}
u_1^2 -v_1^{p+1} -a N u_1=0 
\end{eqnarray}
or, setting $U:=u_1 -\frac{aN}{2}$ and $V:=v_1$,
\begin{eqnarray}
\label{4.3.5}
U^2 =V^{p+1} +\left( \frac{aN}{2} \right)^2  
\end{eqnarray}
which gives our hyperelliptic model for $D_s$. 

\medskip

Finally, the assertion on the completed local rings at the singularities in the special fiber follows for instance from~\cite{Liu02}, Chapter~10.3,
Proposition~3.48, combined with the fact that the semistable model of $\overline{{\cal M}} ({\cal P}, \Gamma_{\mathrm{ns}} (p))$ over $W'$ is the
pullback of the semistable model over~$W$.~$\Box$
\end{proof}
\medskip

\subsection{Stable model for $\overline{{\cal M}} ({\cal P}, \Gamma_{\mathrm{ns}}^+ (p))$}

Now for curves associated with the {\it normalizer} $\Gamma_{\mathrm{ns}}^+ (p)$ of $\Gamma_{\mathrm{ns}} (p)$.

\medskip

\begin{theo} 
\label{NS+} 
Let $p>3$ be a prime, and let $[\Gamma_{\mathrm{ns}}^+  (p)]$ be the moduli problem over $\Z [1/p]$ associated with $\Gamma_{\mathrm{ns}}^+ (p))$.
Let ${\cal P}$ be as in~Theorem~\ref{NS}, and let
$\overline{{\cal M}} ({\cal P}, \Gamma_{\mathrm{ns}}^+ (p)) =\overline{{\cal M}} ({\cal P}, \Gamma (p))/\Gamma_{\mathrm{ns}}^+ (p)$
be the corresponding compactified fine moduli space. We denote, as in~Theorem~\ref{NS}, by ${\cal S}_{\cal P}$ the number of supersingular points of
$\overline{{\cal M}} ({\cal P}) (\overline{\F}_p)$, and by $W$ a totally ramified extension of $\Z_p^{\mathrm{ur}}$ of degree $(p^2 -1)/2$. 

\medskip

   If $p\equiv 1$ {\rm mod} $4$, then $\overline{{\cal M}} ({\cal P}, \Gamma_{\mathrm{ns}}^+ (p))$ has a semistable model over $W$ whose special fiber is made of two vertical
parts, which are both isomorphic to the Igusa curve $\overline{\cal M} ({\cal P}, {\mathrm{Ig}} (p)/C_4 )_{\overline{\F}_p}$, where $C_4$ denotes the cyclic subgroup 
of order $4$ in $\F_p^*$. Those two parts are linked above each supersingular points of  $\overline{\cal M} ({\cal P})$ by horizontal Drinfeld components. 

\medskip

   If $p\equiv -1$ {\rm mod} $4$, then $\overline{{\cal M}} ({\cal P}, \Gamma_{\mathrm{ns}}^+ (p))$ has a semistable model over $W$ whose special fiber is made of only one 
vertical part, which is isomorphic to the enhanced Igusa curve $\overline{\cal M} ({\cal P}, {\mathrm{Ig}} (p)/\{\pm 1\} )_{\overline{\F}_p}$. That 
vertical part is crossed at all ${\cal S}_{\cal P}$ supersingular points  by a horizontal component.

\medskip 
 
   Wether $p$ is $1$ or $-1$ {\rm mod} $4$, the ${\cal S}_{\cal P}$ horizontal components of the special fiber are copies of some hyperelliptic curve $D^+$ for which 
an affine model is given by
\begin{eqnarray}
\label{DrDinfeldXns0+}
Y^2 =X(X^{\frac{p+1}{2}}  +A_{\mathrm{ns}}  )
\end{eqnarray}
for $A_{\mathrm{ns}} $ in $\overline{\F}_p^*$.\medskip
   
The singular points in the special fiber have local equations either $W[[x,y]]/(xy-\pi_0 )$, if $p\equiv -1 \mod 4$, or $W[[x,y]]/(xy-\pi_0^2 )$ if $p\equiv 1\mod 4$, for $\pi_0$ a 
uniformizer of $W$.

\end{theo}

\medskip

The same caveat as in Remark~\ref{bisrepetitas} (regarding irreducibility of the vertical Igusa parts) is in order here.

\medskip

As before, for a picture of the curve we refer to Figure~\ref{figureNS+}, representing the coarse quotient $X_{\mathrm{ns}} (p)^+$.

\begin{proof}

Use notations as in the above proof of Theorem~\ref{NS}:
our basis $(P_1 ,P_2 )$ of $\F_{p^2}^2$ made of two $\F_{p^2} /\F_p$-conjugate vectors is such that 
$\Gamma_{\mathrm{ns}} (p)  =\{ \left(
\begin{array}{cc}
a^p & 0\\
0 & a
\end{array}
\right) ,
a\in \F_{p^2}^* \}$ with respect to the basis $(P_1 ,P_2 )$. Then the normalizer of $\Gamma_{\mathrm{ns}} (p) $ deprived from  $\Gamma_{\mathrm{ns}} (p) $ is made of all elements 
$$
w_r :=\left(
\begin{array}{cc}
0 & r\\
r^p & 0
\end{array}
\right) 
$$
for $r$ running through $\F_{p^2}^*$. The element $w_1$ leaves stable the $\F_p$-line spanned by $(P_1 +P_2 )$. Up to changing choices, one can assume
that is the line $(1\colon 0)$ chosen in our representative of $(\F_p^* \times \PPP^1 (\F_p ) )/\Gamma_{\mathrm{ns}} (p) $. Therefore $w_1$ maps $(x,(1\colon 0))$ 
to $(-x,(1\colon 0))$, so that it exchanges the two orbits corresponding to our Igusa parts if and only if $-1$ is a non-square in $\F_p^*$. %

\medskip

  Now for the Drinfeld components. One needs to compute the action of $w_r$  for $r$ satisfying $r^{p+1} =-1$. With notations as in~(\ref{u1v1}), one checks that, independently of $r$ and because
of~(\ref{4.3.4}): 
$$
\left\{
\begin{array}{lll}
u_1 \cdot w_r &=&\tilde{\alpha}^{p+1} \cdot w_r =r^{(p+1)p} \tilde{\beta}^{p+1} =- \tilde{\beta}^{p+1} =-\frac{v_1^{p+1}}{u_1} =-u_1 +aN  \hspace{0.5cm} \mathrm{and} \\
v_1 \cdot w_r &=&(r^p \tilde{\beta} )(r \tilde{\alpha} ) =- \tilde{\alpha} \tilde{\beta} =-v_1 
\end{array}
\right.
$$
so that $U= (u_1 -\frac{aN}{2} )$ and $V=v_1$ are mapped to their opposite. Therefore 
\begin{eqnarray}
\label{x1y1}
\left\{
\begin{array}{l}
X := v_1^2 =V^2 \hspace{0.8cm} \mathrm{and} \\
Y :=(u_1 -\frac{aN}{2} )\times  v_1  =U\times V
\end{array}
\right.
\end{eqnarray}
give coordinates for the image of any Drinfeld component in our ${\overline{\cal M}} ({\cal P}, \Gamma_{\mathrm{ns}}^+ (p))^{\mathrm{st}}_{\overline{\F}_p}$. 
From~(\ref{4.3.4}) we then check that a singular model  for any Drinfeld component  can now be given the equation
\begin{eqnarray*}
Y^2 = 
X(X^{\frac{p+1}{2}} +A_{\mathrm{ns}} ).%
\end{eqnarray*}


\medskip

  The proof of the equations of singularities in the special fiber are straightforward and similar to that of Theorem~\ref{NS}. $\Box$

\end{proof}

\subsection{Stable model for ${X}_{\mathrm{ns}} (p)_\Q$} 
\label{Xns}


\begin{figure}
\begin{center}
\begin{picture}(320,200)(-30,-80)


\qbezier(-20,55)(135,85)(235,50)
\put(-40,55){$D_s$}

\qbezier(-20,35)(165,55)(235,20)
\put(-40,35){$D_s$}

\qbezier(-20,-45)(115,-35)(235,-70)
\put(-40,-45){$D_s$}


\qbezier(60,105)(65,0)(45,-70)
\put(35,-90){${\mathrm{Ig}} (p)_1$}

\qbezier(140,105)(135,0)(155,-70) 
 \put(147,-90){${\mathrm{Ig}} (p)_d$}


\put(40,27){$\cdot$}
\put(40,5){$\cdot$}
\put(40,-17){$\cdot$}
\put(40,-36){$\cdot$}

\put(160,27){$\cdot$}
\put(160,5){$\cdot$}
\put(160,-17){$\cdot$}
\put(160,-36){$\cdot$}


\put(270,-35){$g(X_0 (p)) +1$} 
\put(275,-50){copies}

\put(265,-22){\vector(-1,3){20}}

\put(265,-28){\vector(-1,2){17}}

\put(265,-32){\vector(-2,-3){20}}


\end{picture}
\end{center}
\caption{Special fiber of ${X}_{\mathrm{ns}} (p)_{\overline{\F}_{p}}^{\mathrm{st}}$}
\label{figureNS}
\end{figure}


Now we deal with the case of pure level $p$ non-split Cartan. We therefore assume the additional level structure ${\cal P}$
is Galois, and take the quotient of our fine modular curves by its Galois group to produce the desired coarse moduli spaces.

\begin{theo} 
\label{NS0} 
For $p>3$ a prime, let $X_{\mathrm{ns}} (p)$ be the modular curve associated with a non-split Cartan subgroup in level $p$. Let ${\cal S}=g(X_0 (p)) +1$ be the 
number of supersingular $j$-invariants in $\F_{p^2}$, where $g(X_0 (p))$ is the genus of $X_0 (p)$. Let $W$ be a totally ramified extension of $\Z_p^{\mathrm{ur}}$ of 
degree $(p^2 -1)/2$, as in Theorem~\ref{NS}. Then $X_{\mathrm{ns}} (p)$ has a semistable model over $W$ whose special fiber is made of two vertical irreducible 
components, which are linked at supersingular points by ${\cal S}$ horizontal components, cf. Figure~\ref{figureNS}. The toric part of its jacobian therefore has 
dimension ${\cal S}-1=g(X_0 (p))$.
 
\medskip 
 
 Both vertical irreducible components, call them ${\mathrm{Ig}} (p)_1$ and ${\mathrm{Ig}} (p)_d$, are isomorphic to the coarse Igusa 
 curve $\overline{M} ({\mathrm{Ig}} (p)/\{ \pm 1\} )_{\overline{\F}_p} )$.

\medskip

  The ${\cal S}$ horizontal (Drinfeld) components are all hyperelliptic smooth curve $D_s$ for which an affine model  is given by 
\begin{eqnarray}
\label{DrinfeldXns}
U^2 =V^{\frac{p+1}{e(s)}} +A_{\mathrm{ns}} 
\end{eqnarray}
for some $A_{\mathrm{ns}} $ in $\overline{\F}_p^*$, and $e(s)$ is the order of the geometric automorphism group ${\mathrm{Aut}}_{\overline{\F}_p} (s)/\{\pm 1\}$
(which we recall to be $1$ except when the $j$ -invariant at $s$ is $j\equiv 1728$ or $0$ {\rm mod} $p$, where $e=2$ or $3$ respectively).

\medskip

  The singular points have local equations $W[[x,y]]/(xy-\pi_0^{e(s)} )$, for $\pi_0$ a uniformizer of $W$. 

%
%
\end{theo}


\begin{proof}
After Theorem~\ref{NS}, what remains to do is, assuming ${\cal P}$ is Galois with Galois group $G$, to
take the quotient of $\overline{\cal M} ({\cal P}, \Gamma_{ns}) (p)^{\mathrm{st}}$ by $G$.
The stabilizers in $G$ have order $1, 2, 3, 4$ or $6$, hence are prime to $p$, so the only thing to watch out is what happens on the locus 
of extra-automorphisms,
that is, on Drinfeld components associated with supersingular $j$ invariants equal to $1728$ or $0$. It then follows from Section~\ref{0-1728} that the
exceptional automorphism $[i]$ (respectively, $[\zeta ]$) maps the parameters $\alpha$ and $\beta$ to $i\alpha$ and $i\beta$ (respectively,
$\zeta \alpha$ and $\zeta \beta$). Keeping track of those transformations through the computations of equations~(\ref{4.3.2}) to~(\ref{4.3.5})
shows that equations for the relevant quotients Drinfeld components have shape as given in~(\ref{DrinfeldXns}).   $\Box$
\end{proof}

\subsection{Semistable model for ${X}_{\mathrm{ns}}^+ (p)_{\Q}$}
\label{Xns+}

\begin{figure}
\begin{center}
\begin{picture}(320,200)(-30,-80)


\qbezier(-58,55)(5,85)(95,50)
\put(-73,55){$D_s^+$}

\qbezier(-58,35)(5,55)(95,20)
\put(-73,35){$D_s^+$}

\qbezier(-58,-45)(5,-35)(95,-70)
\put(-74,-45){$D_s^+$}


\qbezier(-40,105)(-35,0)(-45,-70)

\qbezier(40,105)(35,0)(55,-70) 


\put(-20,-85){$p\equiv 1$ {\rm mod} $4$}


\put(-50,27){$\cdot$}
\put(-50,5){$\cdot$}
\put(-50,-17){$\cdot$}
\put(-50,-36){$\cdot$}

\put(50,27){$\cdot$}
\put(55,5){$\cdot$}
\put(57,-17){$\cdot$}
\put(60,-36){$\cdot$}


\put(125,-35){$g(X_0 (p)) +1$} 
\put(136,-50){copies}

\put(120,-22){\vector(-1,3){20}}
 
\put(120,-28){\vector(-1,2){19}}

\put(120,-32){\vector(-2,-3){20}}


\put(185,-16){\vector(1,4){18}}
 
\put(187,-25){\vector(1,3){19}}

\put(189,-32){\vector(2,-1){20}}


\qbezier(210,58)(245,65)(295,50)
\put(305,50){$D_s^+$}

\qbezier(212,35)(245,45)(295,30)
\put(305,30){$D_s^+$}

\qbezier(218,-45)(245,-35)(295,-45)
\put(305,-50){$D_s^+$}

\put(222,27){$\cdot$}
\put(221,5){$\cdot$}
\put(222,-17){$\cdot$}
\put(224,-36){$\cdot$}


\qbezier(240,105)(235,0)(245,-70)
%


\put(215,-85){$p\equiv -1$ {\rm mod} $4$}


\end{picture}
\end{center}
\caption{Special fiber on $\overline{\F}_{p}$ of the {\em semistable} model ${X}_{\mathrm{ns}}^+ (p)$, depending on $p\equiv \pm1$ {\rm mod} $4$}
\label{figureNS+}
\end{figure}


\begin{theo}  
\label{NS0+}

Let $p>3$ be a prime, and keep same notations as in~Theorem~\ref{NS0}. Let $w$ be the involution of the curve $X_{\mathrm{ns}} (p)$ associated 
with the quotient of the normalizer of the non-split Cartan subgroup by the Cartan itself, and $X_{\mathrm{ns}}^+ (p) :=X_{\mathrm{ns}} (p)/w$ the 
quotient curve. 
 
   Then in the special fiber of the stable model $X_{\mathrm{ns}} (p)^{\mathrm{st}}$ given in Theorem~\ref{NS0}, 
$w$ fixes horizontal components, and it switches the two vertical ones if and only if $p\equiv -1$ {\rm mod} $4$. The dual graph of its special fiber is therefore topologically 
the same as that of $X_{\mathrm{ns}} (p)$ if $p\equiv 1$ {\rm mod} $4$, or has trivial homology if $p\equiv -1$ {\rm mod} $4$, cf. Figure~\ref{figureNS+}. 
The vertical components are either both isomorphic to the Igusa curve $\overline{M} ({\mathrm{Ig}} (p)/C_4 )_{\overline{\F}_p}$ (where $C_4$ denotes the cyclic 
subgroup of order $4$ in $\F_p^*$), in case $p\equiv 1$ {\rm mod} $4$, or, if $p\equiv -1$ {\rm mod} $4$, is isomorphic to $\overline{M} ({\mathrm{Ig}} (p)/
\{ \pm 1\} )_{\overline{\F}_p} $. The toric rank is precisely 
\begin{itemize}
\item $t=\frac{p-13}{12}$ if $p\equiv 1$ {\rm mod} $12$;
\item  $t=\frac{p-5}{12}$ if $p\equiv  5$ {\rm mod} $12$;
\item  else $t=0$. 
\end{itemize}

\medskip

   The ${\cal S}=g(X_0 (p))+1$ horizontal components are hyperelliptic curve $D_s^+$ for which an affine model is given,
if the supersingular $j$-invariant attached to $D^+_s$ is neither $0$ or $1728$, by
\begin{eqnarray}
\label{DrDinfeldXns+}
Y^2 =X(X^{\frac{p+1}{2}}  +A_{\mathrm{ns}}  )
\end{eqnarray}
for some $A_{\mathrm{ns}} $ in $\overline{\F}_p^*$. In the case the $j$-invariant is $0$, $D_{0}^+$ has a model
\begin{eqnarray}
\label{DrDinfeldXns00+}
Y^2 =X(X^{\frac{p+1}{6}}  +A_{\mathrm{ns}}  ).
\end{eqnarray}
If the supersingular $j$-invariant is $1728$, $D_{1728}^+$ is just a projective line ${\PPP^1}_{\overline{\F}_p}$.  
   
\medskip
   
The singular points in the special fiber have local rings $W[[x,y]]/(xy-\pi_0^{e(s)} )$ if $p\equiv -1$ {\rm mod} $4$, else they are $W[[x,y]]/(xy-\pi_0^{2e(s)} )$,
where $e(s)$ denotes as usual the order of the geometric automorphism group ${\mathrm{Aut}}_{\overline{\F}_p} (s)/\{\pm 1\}$,
and $\pi_0$ is a uniformizer of $W$ (e.g. $(1-\zeta_p )^{2/(p +1)}$).

\end{theo} 
\begin{rema}
\label{seulementsemi}
{\rm Note that in the cases where projective lines are showing-up as Drinfeld components, the model we obtain is only semistable, and 
needs contracting the only rational curve in order to become stable.}
\end{rema}

\begin{proof}
We use Theorem~\ref{NS+}, applying similar arguments as in the proof of Theorem~\ref{NS0}. Again the only delicate point is to follow the effect of
exceptional automorphisms on relevant Drinfeld components, associated with some supersingular elliptic curve $E_0$. So write
$\zeta_n$ ($n=4$ or $6$) for our generator of $\mu_n (\F_{p^2}) ={\mathrm{Aut}}_{\F_{p^2}} (E_0 )$. Keeping track of the action of $\zeta_n$ 
on parameters $\alpha$, $\beta$ as given in Section~\ref{0-1728}, and the subsequent parameters given in~(\ref{x1y1}), one sees that
$[\zeta_n ]$ maps $X$ to $\zeta_n^4 X$ and $Y$ to $\zeta_n^2 Y$. (One readily checks that equations~(\ref{DrDinfeldXns0+}) are preserved.) 
In the case $n=6$, parameters for the quotient Drinfeld component by the action of $[\zeta_6 ]$ are clearly ${\cal X}:=X^3$ and ${\cal Y}:=XY$, so one deduces
from equation~(\ref{DrDinfeldXns0+}): $Y^2 =X(X^{(p+1)/2} +A_{\mathrm{ns}})$ that a model for that quotient Drinfeld component is ${\cal Y}^2 ={\cal X} ({\cal X}^{(p+1)/6} 
+A_{\mathrm{ns}})$. When $n=4$ however, parameters for the quotient Drinfeld component by the action of $[i]$ are ${\cal X}:=X$ and ${\cal Y}:=Y^2$. 
From~(\ref{DrDinfeldXns0+}) we therefore see that taking quotient by the action of~$[i]$ gives the projective line, $[i]$ being the hyperelliptic involution of $D_{1728}$. $\Box$  
\end{proof}

\begin{coro}
Let $p=11$ or $p>13$, and let $J=\mathrm{Jac} (X_{\mathrm{ns}}^+ (p) )$ be the jacobian of $X_{\mathrm{ns}}^+ (p)$ over the fraction field
of a totally ramified extension $W$ of $\Z_p^{\mathrm{ur}}$ of degree $(p^2 -1)/2$. Set $n:={\mathrm{num}} ((p-1)/12)$ and 
let ${\cal S}:=g(X_0 (p)) +1$ be the number of supersingular points in characteristic $p$. Then the 
N\'eron model of $J$ over $W$ has a component group at the special fiber which is isomorphic to 
$$
\left(  \Z /4n\Z \right) \times \left(  \Z /4\Z \right)^{{\cal S}-2}
$$
if $p\equiv 1\mod 4$, and is trivial if $p\equiv -1 \mod 4$.
\end{coro}

\begin{proof}
This follows from Theorem~\ref{NS0+} together with classical results of Raynaud (which describe the component group of the Jacobian in terms of
the intersection matrix of the special fiber of the curve, see Theorem~9.6.1 of~\cite{BLR}).  Notice that for $p\equiv 1 \mod 4$ our theorem
shows that the metrized dual graph of $X_{\mathrm{ns}}^+ (p)$ above $p$ is exactly the same as that 
of $X_0 (p)$ at the special fiber of some totally ramified extension of $\Q_p$ with degree $4$; then one can use for instance~\cite{LF16}, 
Proposition~2.11. Of course one could similarly write the component group over any (ramified) extension of ${\mathrm{Frac}} (W)$.~$\Box$
\end{proof}

\begin{rema}
{\rm 
As a safety check one can compute that, assuming $p\equiv 5$ {\rm mod} $12$ to fix ideas, 
the genus of the Igusa components $\overline{M} ({\mathrm{Ig}} (p)/C_4 )_{\overline{\F}_p}$
is $(p-5)(p-17)/96$, that of the $(p+7)/12$ Drinfeld components is $(p-1)/4$ for all but the $j=0$ one, for which
it is $(p-5)/12$, and the toric rank is $(p-5)/12$ too. The total genus therefore sums up to  
$$
2\, \frac{(p-5)(p-17)}{96} +( \frac{p+7}{12} -1) \cdot \frac{p-1}{4} +\frac{p-5}{12} +\frac{p-5}{12} =\frac{(p-5)^2}{24}
$$  
which indeed is the known genus of $X_{\mathrm{ns}}^+ (p)$ as a Riemann surface (check for instance~\cite{Ma77}  p.~117).

\medskip

Over $\Z [1/p]$, a nice modular interpretation of $X_{\mathrm{ns}}^+ (p)$ has been given in~\cite{RW17}. It is however hard to see what 
survives of it here above $p$. 

\medskip

We remark that when $p\equiv -1$ {\rm mod} $4$, the N\'eron model of the jacobian of $X_{\mathrm{ns}}^+ (p)$ over $W$ gives an example of an abelian
scheme which yet has ``bad reduction'' above $p$ as a polarized abelian variety, in the sense that it decomposes as the product
of abelian varieties with the induced (reducible) product polarization.

\medskip

Note also that we could have derived the mere toric (and abelian) dimensions of stable models for $X_{\mathrm{ns}}^+ (p)$ from the corresponding
description for $X_{\mathrm{s}}^+ (p)$ as recalled in next section, using Chen isogeny between ${\mathrm{Jac}} 
(X_{\mathrm{s}}^+ (p))^{p-{\mathrm{new}}}$ and ${\mathrm{Jac}} (X_{\mathrm{ns}}^+ (p))$ (cf.~\cite{dSE}, \cite{IC}). We come-back to that 
point in Remark~\ref{level13} below.} 
\end{rema}

\section{Split Cartan structures}
\label{SplitCartan}

In this section we describe the bad fibers of $X_{\mathrm{s}} (p)$ and $X_{\mathrm{s}}^+ (p)$, following the same paths as for the non-split Cartan cases. 
Recall that those models (at least for the split Cartan curves $X_{\mathrm{s}} (p)$, if not their Fricke quotient $X_{\mathrm{s}}^+ (p)$) had already been described in the 
first author's thesis (\cite{Edix89}, \cite{Edix89b}).

\subsection{Stable model for ${{\overline{\cal M}}} ({\cal P}, \Gamma_s (p))$}

\begin{theo} 
\label{S} 
Let $p>3$ be a prime, and let $[\Gamma_{\mathrm{s}} (p)]$ be the moduli problem over $\Z [1/p]$ associated with a split Cartan subgroup $\Gamma_{\mathrm{s}} (p)$ 
in level $p$ (not its normalizer).
Let ${\cal P}$ be a representable moduli problem, which is finite \'etale over $({\mathrm{Ell}})_{/\Z_p}$.  %
Let $\overline{{\cal M}} ({\cal P}, \Gamma_{\mathrm{s}} (p)) =\overline{{\cal M}} ({\cal P}, \Gamma (p))/\Gamma_{\mathrm{s}} (p)$ be the associated compactified 
fine moduli space. We denote by $W$ a totally ramified extension of $\Z_p^{\mathrm{ur}}$ of degree $(p^2 -1)/2$, as in Theorem~\ref{NS}. 

\medskip 

   Then $\overline{{\cal M}} ({\cal P}, \Gamma_{\mathrm{s}} (p))$ has a semistable model over $W$ whose special fiber is made of four vertical Igusa parts,
which are linked by horizontal Drinfeld components above each supersingular points via the projection $\overline{{\cal M}} ({\cal P}, \Gamma_{\mathrm{s}} (p)) 
\to \overline{{\cal M}} ({\cal P}) $.

\medskip 
 
 The two central parts are isomorphic to enhanced quotients of Igusa curves $\overline{{\cal M}} ({\cal P},{\mathrm{Ig}} (p)/\{ \pm 1\} )_{\overline{\F}_p} )$.
We call them ${\mathrm{Ig}} (p)_1$ and ${\mathrm{Ig}} (p)_d$ (for $d$ some non-square in $\F_p^*$). The two outer vertical parts are simply copies 
of $\overline{{\cal M}} ({\cal P})$.

\medskip

  If ${\cal S}_{\cal P}$ is the number of supersingular points in $\overline{{\cal M}} ({\cal P})$, the ${\cal S}_{\cal P}$ horizontal (Drinfeld) components are all copies of some
hyperelliptic curves for which an affine model is 
\begin{eqnarray}
\label{rinfeldXns}
U^2 =V^{p+1} +A_{\mathrm{s}}  
\end{eqnarray}
for some non-zero $A_{\mathrm{s}}$ in $\overline{\F}_p^*$. 

\medskip

    The double points of the central Igusa components have local rings $W[[x,y]]/(xy-\pi_0 )$, and those on the two rational outer vertical components, 
have local rings $W[[x,y]]/(xy-\pi_0^{(p-1)/2} )$, for $\pi_0$ a uniformizer of $W$.

\end{theo}

\begin{proof}

This is very akin to the proof of Theorem~\ref{NS}. We compute the quotient of the vertical Igusa parts ${\mathrm{Ig}}_{i,P}$ indexed by $\F_p^* \times \PPP^1 (\F_p )$. 
Fixing a split torus 
$$
\Gamma_{\mathrm{s}} (p)=
\{
\left(
\begin{array}{cc}
a & 0\\
0 & b
\end{array}
\right)
, a,b\in \F_p^*
\}$$ 
and writing again $D\simeq \F_p^*$ for the subgroup of diagonal matrices, $\Gamma_{\mathrm{s}} (p)$ acts on $\PPP^1 (\F_p )$ via its quotient $\Gamma_{\mathrm{s}} (p)  /
D \simeq \Z/(p-1)\Z$. That action has two fixed points say $(1\colon 0)$ and $(0\colon 1)$, and one orbit of size $p-1$. 
One chooses as representatives for the coset $(\F_p^* \times \PPP^1 (\F_p ) )/{\Gamma_{\mathrm{s}} (p)}$ the four elements $(1,(1\colon 0))$, $(1,(0\colon 1))$, 
$(1,(1\colon 1))$ and $(d,(1\colon 1))$ for $d$ some non-square in $\F_p^*$. The Igusa parts
attached with the first two representatives, have stabilizer ${\Gamma_{\mathrm{s}} (p)}\cap {\mathrm{SL}}_2 (\F_p ) = \{  \left(
\begin{array}{cc}
t & 0\\
0 & t^{-1}
\end{array}
\right),
t\in \F_p^* \}$. The stabilizer of the other two parts is $\{ \pm 1\}$. So two vertical parts are isomorphic to the quotient 
$\overline{\cal M} ({\cal P},{\mathrm{Ig}} (p)/\F_p^*)_{{\overline{\F}_p}} \simeq \overline{\cal M} ({\cal P})_{{\overline{\F}}_p}$,
and two are isomorphic to $\overline{\cal M} ({\cal P}, {\mathrm{Ig}} (p)/\{ \pm 1 \} )_{{\overline{\F}_p}}$. 
This is for the first part of the Theorem.

    Let us deal with the Drinfeld components. Recall (cf.~(\ref{modelDis})) that an equation for them in the bad fiber of the semistable 
model $\overline{\cal M} (\Gamma (p) )^{\mathrm{st}}$ is given by
\begin{eqnarray}
\label{4.3.2.S}
-a =\alpha \beta^p -   \alpha^p \beta
\end{eqnarray}
for some $a$ in $\overline{\F}_p^*$.  The stabilizer in ${\Gamma_{\mathrm{s}} (p)}$ of any component~$D_{i,s}$ 
is ${\Gamma_{\mathrm{s}} (p)}\cap {\mathrm{SL}}_2 (\F_p )$,
its action on coordinates of $D_{i,s}$ is given by $(\alpha ,\beta ) \mapsto (t\alpha ,t^{-1} \beta )$, so coordinates on 
$D_{i,s} /{\Gamma_{\mathrm{s}} (p)}\cap {\mathrm{SL}}_2 (\F_p )$ can be chosen as $(u,v)=(\alpha^{p-1} ,\alpha \beta )$.  
From that, equation~(\ref{4.3.2.S}) becomes
\begin{eqnarray}
\label{4.3.3.s}
v^p  -   u^2 v +a\cdot u=0
\end{eqnarray}
and the change of variables $(U,V) :=(uv -a/2,v)$ yields:
$$
U^2 =V^{p+1} +\frac{a^2}{4} 
$$
as a hyperelliptic model for $D_s$. The assertion about the thickness of singularities follows from similar arguments as those in the proof of Theorem~\ref{NS}.  $\Box$


\end{proof}

\subsection{Stable model for ${\overline{{\cal M}}} ({\cal P}, \Gamma_s^+ (p))$}

\begin{theo} 
\label{S+} 
Let $p>3$ be a prime, and let $[\Gamma_{\mathrm{s}}^+  (p)]$ be the moduli problem over $\Z [1/p]$ associated with the normalizer $\Gamma_{\mathrm{s}}^+  (p)$ 
of a split Cartan subgroup in level $p$. Let ${\cal P}$ a moduli problem as in~Theorem~\ref{NS}, and let $\overline{{\cal M}} ({\cal P}, \Gamma_{\mathrm{s}}^+ (p))  
=\overline{{\cal M}} ({\cal P}, \Gamma (p)) /\Gamma_{s}^+  (p)$ be the corresponding 
compactified fine moduli space. Let $W$ be a totally ramified extension of $\Z_p^{\mathrm{ur}}$ of degree $(p^2 -1)/2$, and ${\cal S}_{\cal P}$ be the number of 
supersingular points of $\overline{{\cal M}} ({\cal P}) (\overline{\F}_p)$.
 
\medskip
 
    If $p\equiv 1$ {\rm mod} $4$, then $\overline{{\cal M}} ({\cal P}, \Gamma_{\mathrm{s}}^+ (p))$ has a semistable model over $W$ whose special fiber is made of three vertical 
parts. Two neighbor vertical parts are isomorphic to the enhanced Igusa curve $\overline{\cal M} ({\cal P}, {\mathrm{Ig}} (p)/C_4 )_{\overline{\F}_p}$, where $C_4$ 
denotes the cyclic subgroup of order $4$ in $\F_p^*$. One outer part is a copy of $\overline{{\cal M}} ({\cal P})$. 
Those three parts are linked above supersingular points of $\overline{{\cal M}} ({\cal P})$ by ${\cal S}_{\cal P}$ horizontal components, cf. first case of Figure~\ref{figureS+}. 

\medskip

   If $p\equiv -1$ {\rm mod} $4$, then $\overline{{\cal M}} ({\cal P}, \Gamma_{\mathrm{s}}^+ (p))$ has a semistable model over $W$ whose special fiber is made of only two
vertical parts. One is isomorphic to the Igusa curve $\overline{\cal M} ({\cal P}, {\mathrm{Ig}} (p)/\{\pm 1\} )_{\overline{\F}_p}$. The second 
vertical part is again a copy of $\overline{{\cal M}} ({\cal P})$. Those 
components are linked above supersingular points  of $\overline{{\cal M}} ({\cal P})$ by ${\cal S}_{\cal P}$ horizontal components, cf. second case of Figure~\ref{figureS+}.

\medskip 
 
   Wether $p$ is $1$ or $-1$ {\rm mod} $4$, the ${\cal S}_{\cal P}$ horizontal components $D_s^+$ of the special fiber are copies of some hyperelliptic curve for 
which an affine model is given by

\begin{eqnarray}
\label{DrinfeldXs0^+}
Y^2 =X(X^{\frac{p+1}{2}}  +A_{\mathrm{s}}  )
\end{eqnarray}
for some $A_{\mathrm{s}}$ in $\overline{\F}_p^*$.

\medskip

  Double points on the trivial vertical part (which is a copy of $\overline{{\cal M}} ({\cal P})$) in the special fiber have local rings $W[[x,y]]/(xy-\pi_0^{(p-1)/2} )$, where $\pi_0$ is 
some uniformizer of $W$. As for the genuine Igusa components, singularities have rings $W[[x,y]]/(xy-\pi_0^2 )$, if $p\equiv 1$ {\rm mod} $4$, or $W[[x,y]]/(xy-\pi_0 )$ 
if $p\equiv -1$ {\rm mod} $4$.

\end{theo} 

\begin{proof}

This is again also very similar to the proof of Theorem~\ref{NS+}. We take the further quotient of the curve $\overline{{\cal M}} ({\cal P}, \Gamma_{\mathrm{s}} (p)) $
by the normalizer $ \Gamma_{\mathrm{s}} (p)^+$. 
Fricke's involution $w$ is here given by the set $\{
w_{a,b} :=\left(
\begin{array}{cc}
0 & a\\
b & 0
\end{array}
\right)
,  
a,b\in \F_p^*
\} .$
Therefore $w$ switches the two outer vertical parts. As for the central ones, their representatives $(1,(1\colon 1))$ and $(d,(1\colon 1))$ are mapped to
$(-1,(1\colon 1))$ and $(-d,(1\colon 1))$ respectively, by $w$. It follows that the components ${\mathrm{Ig}} (p)_1$ and ${\mathrm{Ig}} (p)_d$ of
Theorem~\ref{S} are switched if and only if $-1$ is a non-square in $\F_p^*$. 

With notations as in~(\ref{4.3.3.s}) one checks that, for $w_{t,-t^{-1}}$ in ${\mathrm{SL}}_2 (\F_p )$: 
$$
(u,v) \cdot w_{t,-t^{-1}} =(\alpha^{p-1} , \alpha \beta )\cdot w_{t,-t^{-1}} =((t\beta)^{p-1} , -\alpha \beta ) =(\frac{v^{p-1}}{u} ,-v) =(u-\frac{a}{v},-v)
$$
so the coordinates $U:=uv-a/2$ and $V:=v$ and mapped to their opposite by $w_{t,-t^{-1}}$. Therefore  
\begin{eqnarray}
\label{x2y2}
\left\{
\begin{array}{l}
X :=V^2    \\
Y := UV 
\end{array}
\right.
\end{eqnarray}
are coordinates for the image of any Drinfeld component in our $X_{\mathrm{s}}^+ (p)_{\overline{\F}_p}$, and we conclude as in the proof of
Theorem~\ref{NS+}. 
$\Box$

\end{proof}

\subsection{Stable model for $X_s (p)_\Q$}

Now for the coarse case.

\begin{theo} 
\label{S0} 
For $p>3$ a prime, let $X_{\mathrm{s}} (p)$ be the modular curve associated with a split Cartan subgroup $\Gamma_{\mathrm{s}} (p)$ in level $p$.
Let ${\cal S}=g(X_0 (p)) +1$ be the number of supersingular $j$-invariants in characteristic $p$, where $g(X_0 (p))$ is the genus of $X_0 (p)$. 
Let $W=\Z_p^{\mathrm{ur}} [(1-\zeta_p )^{2/(p +1)} ]$ be as in Theorem~\ref{NS}. Then $X_{\mathrm{s}} (p)$ has 
a semistable model over $W$ whose special fiber is made of four vertical irreducible components, which are linked in ${\cal S}$ points by ${\cal S}$ horizontal components, 
cf. Figure~\ref{figureS}. The toric part of its jacobian has therefore dimension $3({\cal S}-1)=3g(X_0 (p))$.
 
\medskip 
 
 The two central vertical components, call them ${\mathrm{Ig}} (p)_1$ and ${\mathrm{Ig}} (p)_d$, are isomorphic to the quotient coarse Igusa curve $\overline{M} 
({\mathrm{Ig}} (p)/\{ \pm 1\} )_{{\overline{\F}}_p}$. The two outer vertical components are projective lines.  

\medskip

  The ${\cal S}$ horizontal Drinfeld components are all hyperelliptic smooth curves for which an affine model is given by 
\begin{eqnarray}
\label{DrinfeldXs0}
U^2 =V^{\frac{p+1}{e(s)}} +A_{\mathrm{s}}  
\end{eqnarray}
for some $A_{\mathrm{s}}$ in $\overline{\F}_p^*$, and $e(s)={\mathrm{Card}} ( {\mathrm{Aut}}_{\overline{\F}_p} 
(s)/\{ \pm 1\} )$.

\medskip
  
   Singular points on the rational vertical components have local rings $W[[x,y]]/(xy-\pi_0^{e(s)(p-1)/2} )$ and those on Igusa components have rings 
$W[[x,y]]/(xy-\pi_0^{e(s)} )$, for $\pi_0$ a uniformizer of $W$.

%
%
\end{theo}

(Note that, abusing a bit notations, we have used the same labels ${\mathrm{Ig}} (p)_1$ and ${\mathrm{Ig}} (p)_d$ 
as in Theorem~\ref{S}. Note also that equations~(\ref{DrinfeldXs0}) define genuinely hyperelliptic curves only 
when $p$ is not to small.)  

\begin{proof}
Here we parallel the proof of Theorem~\ref{NS0}.  Indeed Theorem~\ref{S} shows that we only need assume ${\cal P}$ is Galois 
with group $G$, and take the quotient of our semistable model $\overline{\cal M} ({\cal P}, \Gamma_{s} (p))^{\mathrm{st}}$ by $G$. Then we 
check what happens on the locus of extra-automorphisms, that is, on Drinfeld components associated 
with supersingular $j$ invariants equal to $1728$ or $0$. Section~\ref{0-1728} shows that the exceptional automorphism 
$[\zeta_n ]$ (for $n=4$ or $6$) maps the parameters $\alpha$ and $\beta$ to $\zeta_n \alpha$ and $\zeta_n \beta$ 
respectively. Keeping track of those transformations through the computations of equations~(\ref{4.3.3.s}) and around,
and doing the math, shows that the relevant quotients Drinfeld components are indeed 
given by~(\ref{DrinfeldXs0}).   $\Box$

\end{proof}


\begin{figure}
\begin{center}
\begin{picture}(320,200)(-30,-80)


\qbezier(-20,55)(135,85)(235,50)
\put(-40,55){$D_s$}

\qbezier(-20,35)(165,55)(235,20)
\put(-40,35){$D_s$}

\qbezier(-20,-45)(115,-35)(235,-65)
\put(-40,-45){$D_s$}


\qbezier(80,105)(85,0)(75,-70)
\put(55,-90){${\mathrm{Ig}} (p)_1$}

\qbezier(130,105)(124,0)(140,-70) 
 \put(137,-90){${\mathrm{Ig}} (p)_d$}


\qbezier(10,105)(15,0)(-5,-70)
\put(55,-90){${\mathrm{Ig}} (p)_1$}

\qbezier(195,105)(189,0)(212,-76) 
 \put(137,-90){${\mathrm{Ig}} (p)_d$}


\put(40,27){$\cdot$}
\put(40,5){$\cdot$}
\put(40,-17){$\cdot$}
\put(40,-36){$\cdot$}

\put(160,27){$\cdot$}
\put(160,5){$\cdot$}
\put(160,-17){$\cdot$}
\put(160,-36){$\cdot$}


\put(270,-35){$g(X_0 (p)) +1$} 
\put(275,-50){copies}

\put(265,-22){\vector(-1,3){20}}

\put(265,-28){\vector(-1,2){17}}

\put(265,-32){\vector(-2,-3){20}}


\end{picture}
\end{center}
\caption{Special fiber of ${X}_{\mathrm{s}} (p)$}
\label{figureS}
\end{figure}

 
\medskip

\subsection{Stable model for $X_s^+ (p)_\Q$}

\begin{theo}  
\label{S0+}

Let $p>3$ be a prime, and use the same notations as in~Theorem~\ref{S} above. Let $w$ be the involution of the curve 
$X_{\mathrm{s}} (p)$ defined by the action of the normalizer $\Gamma_{\mathrm{s}}^+ (p)$, and let $X_{\mathrm{s}}^+ 
(p) :=X_{\mathrm{s}} (p)/w$ be the quotient curve. Let $W=\Z_p^{\mathrm{ur}} [(1-\zeta_p )^{2/(p +1)}]$ as in Theorem~\ref{NS}. 
 
   Then in the special fiber over $W$, $w$ leaves the horizontal components of $X_{\mathrm{s}} (p)$ stable and exchanges the two outer vertical (rational) 
components. It switches the two central vertical ones if $p\equiv -1$ {\rm mod} $4$, else it leaves them stable. The special fiber 
of $X_{\mathrm{s}}^+ (p)^{\mathrm{st}}$ over $W$ therefore has a dual graph as in  Figure~\ref{figureS+}. Its toric rank is explicitely 
\begin{itemize}
\item $t=\frac{p-13}{6}$ if $p\equiv 1$ {\rm mod} $12$;
\item  $t=\frac{p-5}{6}$ if $p\equiv  5$ {\rm mod} $12$;
\item  $t=\frac{p-7}{12}$ if $p\equiv 7$ {\rm mod} $12$;
\item  $t=\frac{p+1}{12}$, if $p\equiv  11$ {\rm mod} $12$. 
\end{itemize}

\medskip

One vertical component of $X_{\mathrm{s}}^+ (p)^{\mathrm{st}}$ is therefore a projective line.  Each of the two other 
vertical components, in the case $p\equiv 1$ {\rm mod} $4$, is isomorphic to the quotient coarse Igusa curve 
$\overline{M} ({\mathrm{Ig}} (p)/C_4 )_{{\overline{\F}}_{p}}$, for $C_4$ the scalar subgroup of order $4$. 
When $p\equiv -1$ {\rm mod} $4$, the remaining non-rational vertical component is $\overline{M} 
({\mathrm{Ig}} (p)/\{ \pm 1\} )_{{\overline{\F}}_{p}}$.  

\medskip

   The ${\cal S}$ Drinfeld horizontal components, above supersingular invariants different from $0$ and $1728$, are hyperelliptic 
curves for which an affine model is given by
\begin{eqnarray}
\label{DrDrinfeldXs^+}
Y^2 =X(X^{\frac{p+1}{2}}  +A_{\mathrm{s}}  )
\end{eqnarray}
for some $A_\mathrm{s}$ in $\overline{\F}_p^*$.  If the supersingular invariant is $0$, $D_{0}^+$ has a model
\begin{eqnarray}
\label{DrDinfeldXs0+}
Y^2 =X(X^{\frac{p+1}{6}}  +A_{\mathrm{s}}  ).
\end{eqnarray}
If the supersingular invariant is $1728$, then $D_{1728}^+$ is just a projective line ${\PPP^1}_{{\overline{\F}}_{p}}$.

\medskip
   
   Double points on the rational vertical component have rings $W[[x,y]]/(xy-\pi_0^{e(s)(p-1)/2} )$, for 
$e(s)={\mathrm{Card}} ( {\mathrm{Aut}}_{\overline{\F}_p} (s)/\{ \pm 1\} )$ and $\pi_0$ a uniformizer of $W$. 
As for Igusa components, singularities in the special fiber have local rings $W[[x,y]]/(xy-\pi_0^{2e(s)} )$ if $p\equiv 1$ 
{\rm mod} $4$, and $W[[x,y]]/(xy-\pi_0^{e(s)} )$ if $p\equiv -1$ {\rm mod} $4$.

\end{theo}

\begin{proof}
This time what we mimic is Theorem~\ref{NS0+}: use Theorem~\ref{S+}, applying similar arguments as in the proof of Theorem~\ref{S0}. $\Box$  
\end{proof}

\begin{figure}
\begin{center}
\begin{picture}(320,200)(-30,-80)


\qbezier(-58,55)(5,85)(95,50)
\put(-73,55){$D_s^+$}

\qbezier(-58,35)(5,55)(95,20)
\put(-73,35){$D_s^+$}

\qbezier(-58,-45)(5,-35)(95,-70)
\put(-74,-45){$D_s^+$}


\qbezier(-40,105)(-35,0)(-45,-70)

\qbezier(40,105)(35,0)(55,-70) 


\qbezier(10,105)(15,0)(11,-70)


\put(-20,-85){$p\equiv 1$ {\rm mod} $4$}


\put(-50,27){$\cdot$}
\put(-50,5){$\cdot$}
\put(-50,-17){$\cdot$}
\put(-50,-36){$\cdot$}

\put(50,27){$\cdot$}
\put(55,5){$\cdot$}
\put(57,-17){$\cdot$}
\put(60,-36){$\cdot$}


\put(125,-35){$g(X_0 (p)) +1$} 
\put(136,-50){copies}

\put(120,-22){\vector(-1,3){20}}
 
\put(120,-28){\vector(-1,2){19}}

\put(120,-32){\vector(-2,-3){20}}


\put(185,-16){\vector(1,4){18}}
 
\put(187,-25){\vector(1,3){19}}

\put(189,-32){\vector(2,-1){20}}


\qbezier(210,58)(245,65)(295,50)
\put(305,50){$D_s^+$}

\qbezier(212,35)(245,45)(295,30)
\put(305,30){$D_s^+$}

\qbezier(218,-45)(245,-35)(295,-45)
\put(305,-50){$D_s^+$}

\put(222,27){$\cdot$}
\put(221,5){$\cdot$}
\put(222,-17){$\cdot$}
\put(224,-36){$\cdot$}


\qbezier(240,105)(245,0)(241,-70)


\qbezier(270,105)(275,0)(271,-70)


\put(215,-85){$p\equiv -1$ {\rm mod} $4$}


\end{picture}
\end{center}
\caption{Special fiber of ${X}_{\mathrm{s}}^+ (p)$, depending on $p\equiv \pm1$ {\rm mod} $4$}
\label{figureS+}
\end{figure}

\medskip

\begin{rema}
\label{level13}
{\rm It follows from Chen-Edixhoven's theorem (\cite{IC}, \cite{dSE}) that 
$$
{\mathrm{Jac}} ({X}_{\mathrm{s}}^+ (p) )\sim {\mathrm{Jac}} ({X}_{\mathrm{ns}}^+ (p) )\times {\mathrm{Jac}} ({X}_0 (p) )
$$
so for $p=13$ the split and non-split Cartan curves curves ${X}_{\mathrm{s}}^+ (13)$ and  ${X}_{\mathrm{ns}}^+ (13)$ have 
isogenous jacobians. But in~ \cite{Ba14}, Burcu Baran computed models showing that they even are {\it isomorphic} (for 
some isomorphism which does not seem to have any natural modular interpretation - for instance, the packet of six 
$\Q$-valued CM points and the rational cusp on the former curve are mapped to seven rational CM points on the latter (and 
those sets are proven  in~\cite{BM17} to be the full ${X}_{\mathrm{s}}^+ (13) (\Q )$ and ${X}_{\mathrm{ns}}^+ (13) (\Q )$ 
respectively)). Our two models however look like having different bad fibers: both have one horizontal component, 
but ${X}_{\mathrm{s}}^+ (13)_{\F_{13}}$ has three vertical ones, whereas ${X}_{\mathrm{ns}}^+ (13)_{\F_{13}}$ has only 
two. A closer look however shows that the all vertical components are rational. After contracting the $\PPP^1$s 
only the horizontal component of each model therefore survives, and both happen to be geometrically isomorphic
to the genus-$3$ curves with affine model $Y^2 =X^{8}  +X $. This finally shows that our isomorphic modular curves 
have potentially good reduction everywhere.}    
\end{rema}

\section{Exceptional subgroups}  
\label{Exceptional}

We finally do the computations for modular curves in prime level $p$, associated with linear groups
$\Gamma_{\mathfrak{A}_4} (p)$, $\Gamma_{\mathfrak{S}_4} (p)$ and $\Gamma_{\mathfrak{A}_5} (p)$ 
having projective image the permutation groups~$\mathfrak{A}_4$, $\mathfrak{S}_4$ or~$\mathfrak{A}_5$ respectively
(see~\cite{La76}, Chapter XI, and more specifically \cite{Fe76}, for general facts on those).  

   Things go essentially the same way as for the Cartan cases, to the only exception that equations for the Drinfeld
components are more delicate to write down explicitly. It seems in particular that writing them as {\it quotients},
as we did for the Cartan subgroups, is hardly doable with bare hands. So instead of giving closed expressions 
we describe in next paragraph an algorithmic method to obtain them. Then we review the
other features of special fibers (topology of the dual graph, vertical components...) for the three exceptional cases,
and each time display some numerical examples of those Drinfeld equations. 

\subsection{Computation of Drinfeld components}
\label{Computationofthe}
Starting from the affine equation~(\ref{modelDis}) for the generic Drinfeld component $D$ on $X (p)$, or better the smooth projective 
model
\begin{eqnarray}
\label{modelDismaisprojectif}
x^p y -xy^p =z^{p+1}
\end{eqnarray}
we see that the projection $(x,y,z)\mapsto (x,y)$ presents it as a
$\mu_{p+1}$-covering of the projective line, for $\mu_{p+1}$ the group
of $p+1$st roots of unity, which is ramified precisely above
$\PPP^1 (\F_ p)$.  We also see that $D$ is endowed with an action of
$G:=\SL_2 (\F_p )\times {\bf \mu}_{p+1}$ defined as
\begin{equation}\label{actions_sur_D}
\left( \left( \begin{array}{cc}
a & b \\
c & d 
\end{array}
\right) ,\alpha \right) \cdot ( x, y, z) = (ax+cy , bx+dy , \alpha z)
\end{equation}
(recall the ``transposed'' action of $\SL_2 (\F_p )$ as described in~(\ref{gonalpha})).
Clearly the two actions of $\SL_2 (\F_p )$ and $\mu_{p+1}$ commute. 
The group $G$ does not act faithfully, but its quotient by $\{ \pm 1\} =\mu_2 (\F_p )$ (embedded diagonally), does. 
Therefore if $H$ is any subgroup of $\SL_2 (\F_p )$ (containing $-1$)
we have the commutative diagram
\[
\begin{tikzcd}
D/\mu_2  \ar[r,"\pi"] \ar[d,"q"]  & \PPP^1_{\overline{\F}_p}  \ar[d] \\
D/H   \ar[r]  & (\PPP^1_{\overline{\F}_p})/H  
\end{tikzcd}
\]
where the (smooth) curves $D/\mu_2$ and $(D/\mu_2 )/H=D/H$ on the left-hand
side are endowed with an action of
$\mu_{p+1} /\mu_2 \simeq \mu_{\frac{p+1}{2}}$, the quotients by which
are precisely the projective lines on the right-hand side. This
diagram is co-cartesian by the universal properties of the quotient
morphisms, and cartesian exactly away from the locus in
$(\PPP^1 _{\overline{\F}_p} )/H$ where both maps are
ramified (above such points the fibered product has a $2$-dimensional
tangent space).

Let us first make the quotient $(\PPP^1 _{\overline{\F}_p} )/H$
explicit by giving a rational function
$\phi\colon \PPP^1_{\overline{\F}_p}\to\PPP^1_{\overline{\F}_p}$
that realises it. We can take
\begin{equation}\label{eq:quotient_phi}
\phi (t) =\frac{\prod_{P\in O_1} (t-t(P) )^{\# H_1}}{\prod_{P\in O_2} (t-t(P) )^{\# H_2}},  
\end{equation}
where the $O_i$ are two distinct $H$-orbits of elements of
$\PPP^1(\overline{\F}_p)$, not containing~$\infty$, and the $H_i$ are
their respective isotropy groups. The diagram above now has become
\[
\begin{tikzcd}
D/\mu_2  \ar[r,"\pi"] \ar[d,"q"]  & \PPP^1_{\overline{\F}_p}  \ar[d,"\phi"] \\
D/H   \ar[r,"\overline{\pi}"]  & \PPP^1_{\overline{\F}_p}  
\end{tikzcd}
\]
Via~$\overline{\pi}$, $D/H$ is a $\mu_{\frac{p+1}{2}}$-covering of
$\PPP^1_{\overline{\F}_p}$, hence it can be given a (singular)
equation of shape $u^{(p+1)/2} =f(t)$, with
$\zeta\in \mu_{\frac{p+1}{2}}$ sending $u$ to~$\zeta u$, say, and we
need to spot such an~$f$. We can multiply $f$ by arbitrary non-zero
$(p+1)/2$th powers, so we just need to determine ${\mathrm{div}} (f)$
with coefficients modulo~$(p+1)/2$. For that, we observe that, at each
fixed point, $\zeta\in\mu_{p+1}$ acts on the cotangent space of $D$
by~$\zeta$ (use equation~(\ref{actions_sur_D})). Therefore, at each
fixed point of~$D/\mu_2$, $\zeta\in\mu_{(p+1)/2}$ acts on the
cotangent space by~$\zeta$. Now let $P\in D/\mu_2$ be a fixed point
for~$\mu_{(p+1)/2}$, let $Q:=q(P)$ and let $e=\# H_P$ be the
ramification index of $q$ at~$P$. Then $\zeta\in\mu_{(p+1)/2}$ acts on
the cotangent space at~$Q$ by~$\zeta^e$. We note that
$\# H_P =\# H_{\pi (P)}$, the ramification index of $\phi$ at $\pi (P)$, and that
$v_{\overline{\pi}(Q)}(f){\cdot}(p+1)/2 = v_Q(\overline{\pi}^*f) = v_Q(u){\cdot}(p+1)/2$,
hence $v_{\overline{\pi}(Q)}(f) = v_Q(u)$. It follows that
$\zeta\in\mu_{(p+1)/2}$ sends $u$ to $\zeta^{v_Q(u)e}{\cdot}u$, which
we know to be $\zeta u$ itself. Hence:
\[
  v_{\overline{\pi}(Q)}(f) = v_Q(u) = e^{-1} = (\# H_{\pi (P)} )^{-1}
  \quad\text{in $\Z/((p+1)/2)\Z$}\,.
\]
We finally obtain for our Drinfeld component $D/H$ over
$\overline{\F}_p$ the equation
\begin{eqnarray}
\label{equationD/H}
u^{\frac{p+1}{2}} =\prod_{R\in H\backslash\PPP^1 (\F_p ) ,\, \phi(R)
  \neq \infty} (t -\phi (R))^{1/\# H_R} \,,
\end{eqnarray}
where the product is over a set of representatives $R$ with
$\phi(R)\neq\infty$ for the $H$-orbits of~$\PPP^1(\F_p)$, and where
$1/\# H_R$ is lift in $\Z$ of the inverse of $\# H_R$
in~$(\Z/((p+1)/2)\Z)^\times$.

  (Notice that in all cases below, $H_R$ is the isotropy group of the {\it intersection} of our exceptional groups {\it with  
$\SL_2 (\F_p )$}. In particular, the cases $p\equiv 11$ or $19$ {\rm mod} $12$ in Section~\ref{sectionS4} below (group ${\mathfrak{S}_4}$) 
should cause no worries with respect to the condition that $\# H_R$ is invertible {\rm mod} $(p+1)/2$.)

In next sections we illustrate this method by providing a few numerical
examples, constructing explicitly some $\phi$ and
equations~(\ref{equationD/H}) for each case $H={\mathfrak{A}_4}$,
${\mathfrak{S}_4}$ or ${\mathfrak{A}_5}$.

\subsection{${\mathfrak{A}_4}$}
\label{sectionA4}
We first notice that the fact $\mathfrak{A}_4$ has no subgroup of
index $2$ implies $\mathfrak{A}_4$ in fact belongs to the subgroup
${\mathrm{SL}}_2 (\F_{p}) /\{ \pm 1\}$ of ${\mathrm{GL}}_2 (\F_{p}) /{\F_p^*}$. 
The smallest number field over which the corresponding modular curve has a geometrically connected model is
therefore the quadratic subfield of~$\Q (\mu_p )$.

It follows from~\cite{La76}, proof of Theorem~2.3 on p.~186 of
Chapter~XI, that there are three orbits of elements in
$\PPP^1 (\F_{p^2} )$ with non-trivial isotropy subgroups for the
action of $\mathfrak{A}_4$ in ${\mathrm{PGL}}_2 (\F_{p^2})$. Those
isotropy subgroups have order~$2$, $3$ and~$3$ (cf. case~(iii) of the
Lemma after Theorem~2.3 quoted above); we call them~$G_{2}$, $G_{3,1}$
and~$G_{3,2}$. In $\PPP^1 (\F_{p^2} )$, there is therefore one orbit
of size~$6$ (call it~$O_2$), two of size~$4$ (call them~$O_{3,1}$
and~$O_{3,2}$), and $(p^2 -13)/12$ orbits of size~$12$ (homogeneous
spaces under action of~$\mathfrak{A}_4$).  Restricting that
combinatorics to $\PPP^1 (\F_p )$ sums-up as:

\medskip

\begin{center}

\begin{tabular}{|c||c|c|}
\hline $p$  {\rm mod} $12$ & exceptional orbits in $\PPP^1 (\F_p )$ &   total number $N_p$ of orbits under ${\mathfrak{A}}_4$ \\
\hline   
\hline
1 &      $O_2$ , $O_{3,1}$, $O_{3,2}$   & $(p+23)/12$  \\
\hline
5 &    $O_2$  & $(p+7)/12$ \\
\hline
7 &     $O_{3,1}$, $O_{3,2}$  & $(p+17)/12$ \\
\hline
11 &  none  &  $(p+1)/12$  \\
\hline 
\end{tabular} .

\end{center}

%
%
 
\begin{theo} 
\label{A4} 
Let $p>3$ be a prime, and let $[\Gamma_{\mathfrak{A}_4} (p)]$ be the moduli problem over $\Z [1/p]$ associated with 
$\Gamma_{\mathfrak{A}_4} (p)$. Let ${\cal P}$ be a representable moduli problem, which is finite \'etale over 
$({\mathrm{Ell}})_{/\Z_p}$.  
Let $\overline{{\cal M}} ({\cal P}, \Gamma_{\mathfrak{A}_4} (p)) =\overline{{\cal M}} ({\cal P}, \Gamma (p))/ 
\Gamma_{\mathfrak{A}_4} (p)$ be the associated compactified fine moduli space. 
Let $W$ be a totally ramified extension of $\Z_p^{\mathrm{ur}}$ of degree $(p^2 -1)/2$ (e.g. $W=\Z_p^{\mathrm{ur}} [(1-\zeta_p )^{2/(p+1)}] $)
as in Theorem~\ref{NS}. 
 
\medskip

   Then $\overline{{\cal M}} ({\cal P}, \Gamma_{\mathfrak{A}_4} (p))$ has a semistable model over $W$ whose special 
fiber is made of $2N_p$ vertical Igusa parts (for $N_p \sim p/12$ as in the above array)  %
which are linked by horizontal Drinfeld components above each supersingular points 
of $\overline{{\cal M}} ({\cal P})$ via the projection $\overline{{\cal M}} ({\cal P}, \Gamma_{\mathfrak{A}_4} (p)) \to 
\overline{{\cal M}} ({\cal P})$. The geometric vertical parts
are almost all copies of quotient enhanced Igusa curves $\overline{{\cal M}} ({\cal P}, {\mathrm{Ig}} (p)/
\{ \pm1 \} )_{\overline{\F}_p} $, except that:

\begin{itemize}

\item if $p\equiv 1$ {\rm mod} $12$,  two of them are $\overline{{\cal M}} ({\cal P}, 
{\mathrm{Ig}} (p) /C_4 )_{\overline{\F}_p} $, and four are
$\overline{{\cal M}} ({\cal P}, {\mathrm{Ig}} (p)/ C_6  )_{\overline{\F}_p} $, for $C_*$ a cyclic automorphism group of 
order $*$; 

\item if $p\equiv 5$ {\rm mod} $12$, there are two exceptional Igusa parts, which are $\overline{{\cal M}} ({\cal P}, 
{\mathrm{Ig}} (p)/ C_4 )_{\overline{\F}_p} $;

\item if $p\equiv 7$ {\rm mod} $12$, there are four exceptional Igusa parts, copies of $\overline{{\cal M}} ({\cal P}, 
{\mathrm{Ig}} (p)/ C_6 )_{\overline{\F}_p} $.

\end{itemize}

\medskip

%

Singular points located on components of shape $\overline{{\cal M}} ({\cal P}, {\mathrm{Ig}} (p)/ C_r )_{\overline{\F}_p}$
 have local equations 
 $$
 W[[x,y]]/(xy-\pi_0^r ), r=1, 2 {\mathrm{\ or\ }} 3
 $$
 for $\pi_0$ as usual a uniformizer of $W$ (e.g. $\pi_0 =(1-\zeta_p )^{2/(p+1)}$). 

%
  
%
%
\end{theo}

\begin{proof}
As already remarked, $\mathfrak{A}_4$ has no subgroup of index $2$ so 
that $\mathfrak{A}_4$ in fact belongs to ${\mathrm{SL}}_2 (\F_{p}) /\{ \pm 1\}$, and the image under the determinant of the 
full group $\Gamma_{\mathfrak{A}_4} (p)$ consists of all the squares of $\F_p^*$. Therefore vertical parts
of our quotient curve can be indexed by the set of pairs $\{ ({\cal O}, d ) \}$, where ${\cal O}$ runs through the set
of orbits of $\PPP^1 (\F_p )$ under the action of $\mathfrak{A}_4$, and $d$ runs through $\F_p^*$ modulo squares.

Igusa parts associated with orbits having the generic trivial isotropy group have stabilizer $\pm 1$ in 
$\Gamma_{\mathfrak{A}_4} (p)$, so they are isomorphic to $\overline{{\cal M}} ({\cal P}, 
{\mathrm{Ig}} (p)/\{ \pm 1\} 
)_{\overline{\F}_p }$. As for the Igusa components associated with $G_{2}$ and $G_{3,i}$, they are isomorphic to some $
\overline{{\cal M}} ({\cal P}, 
{\mathrm{Ig}} (p)/C_4 )_{\overline{\F}_p }$ and $\overline{{\cal M}} ({\cal P}, {\mathrm{Ig}} (p)/C_6 )_{\overline{\F}_p }$ 
respectively, for $C_r$ a cyclic
automorphism group of oder $r$. The rest goes as in the proof of the
previous theorems.~$\Box$
\end{proof} 

\medskip

Now for equations of Drinfeld components. It follows from Theorem~6.1
of~\cite{Ch99} that a system of generators for the image of
${\mathfrak{A}_4}$ in ${\mathrm{SL}}_2 (\F_{p}) /\{ \pm 1\}$ can be taken as any
$(S,T)$ with $S$ of (projective) order $3$, $T$ of order $4$, and $ST$
has order $3$, that is, $S$, $T$ and $ST$ are matrices with
determinant $1$ and trace $\pm 1$, $0$ and $\pm 1$ respectively.  One
readily checks with that numerical criterion that ${\mathfrak{A}_4}$
has a model in ${\mathrm{SL}}_2 (\F_{p}) /\{ \pm 1\}$ (although not in
${\mathrm{SL}}_2 (\Z )/\{ \pm 1 \}$). When $p\equiv 1$ {\rm mod} $3$, one can for
instance take
 $$
 S=
\left( 
\begin{array}{cc}
\zeta_3  & 0\\
-1 & \zeta_3^2 
 \end{array}
\right) ; \ 
T=
\left( \begin{array}{cc}
0  &-1\\
1 & 0 
 \end{array}
\right) 
$$
for $\zeta_3$ some primitive third root of unity. In order to
write-down a function $\phi$ as in equation~(\ref{eq:quotient_phi}) it
is enough to find representatives $a_1$ and $a_2$ of two different
$\mathfrak{A}_4$-orbits in $\PPP^1(\F_p)$ and take
$$
\phi (t) =\prod_{g\in {{\mathfrak{A}}_4}} (t-g{\cdot}a_1 ) (t-g{\cdot}a_2 )^{-1} 
$$ 
from which can write explicit forms of
equation~(\ref{equationD/H}). As for numerical examples with $p=13$ or
$p=103$ we compute that:

\medskip

{\bf if $p=13$}, the set $\PPP^1 (\F_{13})$ decomposes in three
orbits: $O_{3,1} =\{ 0, \infty ,9,10 \}$, $O_{3,2} =\{ 1,2,6,12 \}$
and $O_{2} =\{ 3,4,5,7,8,11 \}$ under our action
of~${\mathfrak{A}}_4$, whence, as in~(\ref{eq:quotient_phi}), for
instance a function~$\phi$:
$$
\phi (t)=\frac{[(t-1)(t-2)(t-6)(t-12)]^3}{[ (t-3)(t-4)(t-5)(t-7) (t-8)(t-11)]^2} 
$$      
for which $\phi (\PPP^1 (\F_p ))=\{0, 1, \infty\}$ (the images of the
orbits of~$1$, $\infty$ and~$3$), with ramification indices $3$, $3$
and~$2$, respectively. The inverses in $\Z/7\Z$ of these are $5$, $5$
and~$4$, respectively. We therefore obtain from~(\ref{equationD/H}) the affine
singular model: 
$$
u^7 =t^5 (t-1)^5
$$
for ${\mathfrak{A}}_4$-exceptional Drinfeld components in level~$13$
(with suitable rigidification).

\medskip

{\bf If $p=103$}, we consider for instance the orbits of $0$ and $1$ to obtain
$$
\phi (t)=\frac{[t(t-56)(t-57)]^3}{[ (t-1)(t-10)(t+1)(t-72)]^3} 
$$   
whence 
$$
\phi (\PPP^1 (\F_p )) =\{0, \infty, 3, -3, -1, 22, 39, -14, -39, 10\} 
$$
with respective ramification indices $3,3,1,1,1,1,1,1,1,1$ having
inverses $35, 35, 1,1,1,1,1,1,1,1$ {\rm mod}~$52$, and an equation for ${\mathfrak{A}}_4$-Drinfeld components 
in characteristic $103$ which is
$$
u^{52} =t^{35} (t-3) (t+3)(t+1)(t-22)(t-39)(t+14)(t+39)(t-10)\,.
$$

\subsection{${\mathfrak{S}_4}$}
\label{sectionS4}

Notice that $\mathfrak{S}_4$ belongs to ${\mathrm{SL}}_2 (\F_{p}) /\{ \pm 1\}$ if (and only) if $p\equiv \pm 1$ {\rm mod} $8$, 
cf.~Theorems~1 \& 2 of Feit's appendix in~\cite{La76}, pp.~201 \& 202). If $p\equiv 3$ or $5$ {\rm mod} $8$ then
$\mathfrak{S}_4 \cap {\mathrm{SL}}_2 (\F_{p}) /\{ \pm 1\} =\mathfrak{A}_4$, so  the relevant curve 
$X_{\mathfrak{S}_4} (p)$ is a form of the curve $X_{\mathfrak{A}_4} (p)$ studied in paragraph~\ref{sectionA4} above,
and the former curve does have a geometrically integral model over $\Q$. 
 
    Now~\cite{La76}, proof of Theorem~2.3 on p.~187 of Chapter~XI, gives that there are three orbits of elements 
in $\PPP^1 (\F_{p^2} )$ with non-trivial isotropy subgroups for the action of $\mathfrak{S}_4$ in ${\mathrm{PGL}}_2 
(\F_{p^2})$, and those isotropy subgroups have order $2$, $3$ and $4$ (cf. case (iv) of the Lemma after Theorem~2.3 
quoted above): we shall denote them by $G_2$, $G_3$ and $G_4$ respectively. In $\PPP^1 (\F_{p^2} )$, there is
therefore one orbit of size $12$, one of size $8$, one of size $6$, and $(p^2  -25)/24$ of size $24$ (which are 
homogeneous spaces under action of $\mathfrak{S}_4$). We denote by $O_{2}$, $O_3$ and $O_4$ the exceptional orbits
of order $12$, $8$ and $6$ respectively. Restricting that  combinatorics to~$\PPP^1 (\F_p )$ gives 

\medskip

\begin{center}

\begin{tabular}{|c||c|c|}
\hline $p$  {\rm mod} 24 & exceptional orbits in $\PPP^1 (\F_p )$ &   total number $N_p$ of orbits under ${\mathfrak{S}}_4$ \\
\hline   
\hline
1 & $O_2$ , $O_3$, $O_4$ & $ (p+47) /24$ \\
\hline   
5 & $O_4$ & $(p+19)/24$ \\
\hline
7 & $O_3$ & $(p+17)/24$ \\
\hline 
11 & $O_2$ & $(p+13)/24$ \\
\hline 
13 & $O_3$, $O_4$ & $(p+35)/24$ \\
\hline 
17 & $O_2$, $O_4$ & $(p+31)/24$ \\
\hline 
19 & $O_2$, $O_3$ & $(p+5)/24$ \\
\hline 
23 & none & $(p+1)/24$ \\
\hline 
\end{tabular}  .

\end{center}

\medskip 

\begin{theo} 
\label{S4} 
Let $p>3$ be a prime which is congruent to $\pm 1 \mod 8$, and let $[\Gamma_{\mathfrak{S}_4} (p)]$ be the moduli problem over $\Z [1/p]$ associated with 
$\Gamma_{\mathfrak{S}_4} (p)$.
Let ${\cal P}$ be a representable moduli problem, which is finite \'etale over $({\mathrm{Ell}})_{/\Z_p}$.  
Let $\overline{{\cal M}} ({\cal P}, \Gamma_{\mathfrak{S}_4} (p)) =\overline{{\cal M}} ({\cal P}, \Gamma (p))/ 
\Gamma_{\mathfrak{S}_4} (p)$ be the associated compactified fine moduli space. 
Let $W$ be a totally ramified extension of $\Z_p^{\mathrm{ur}}$ of degree $(p^2 -1)/2$, with uniformizer $\pi_0$, as in Theorem~\ref{NS}. 

\medskip

   Then $\overline{{\cal M}} ({\cal P}, \Gamma_{\mathfrak{S}_4} (p))$ has a semistable model over $W$ whose special 
fiber is made of vertical Igusa parts, which are linked by horizontal 
Drinfeld components above each supersingular points of $\overline{{\cal M}} ({\cal P})$ via the 
projection $\overline{{\cal M}} ({\cal P}, \Gamma_{\mathfrak{S}_4} (p)) \to \overline{{\cal M}} ({\cal P})$.

\medskip

Almost all vertical parts are isomorphic to quotient enhanced Igusa curves $\overline{{\cal M}} ({\cal P}, {\mathrm{Ig}} (p)/\{ 
\pm1 \} )_{\overline{\F}_p} $, except that:
\begin{itemize}
\item if $p\equiv 1\ {\mathrm{mod}}\ 24$, there are six exceptional Igusa part; two are copies of $\overline{{\cal M}} ({\cal P},
{\mathrm{Ig}} (p)/C_4 )_{\overline{\F}_p} $, two are isomorphic to $\overline{{\cal M}} ({\cal P}, {\mathrm{Ig}} (p)/
C_6 )_{\overline{\F}_p}$, and two are $\overline{{\cal M}} ({\cal P}, {\mathrm{Ig}} (p)/ C_8  )_{\overline{\F}_p} $, where 
$C_*$ denotes a cyclic automorphism group of order $*$. The total number of Igusa parts is $2N_p$ (for
$N_p \simeq p/24$ the number of orbits as indicated in the array before our theorem);
\item if $p\equiv 5\ {\mathrm{mod}}\ 24$, there is only one exceptional Igusa part, which is $\overline{{\cal M}} ({\cal P}, 
{\mathrm{Ig}} (p)/ C_4 )_{\overline{\F}_p} $. The total number of Igusa parts is $2N_p -1$; 
\item if $p\equiv 7\ {\mathrm{mod}}\ 24$, there are two exceptional Igusa parts, which are a copies of $\overline{{\cal M}} 
({\cal P}, {\mathrm{Ig}} (p)/ C_6 )_{\overline{\F}_p}$.   The total number of Igusa parts is $2N_p$;
\item if $p\equiv 11\ {\mathrm{mod}}\ 24$, there is no exceptional Igusa part. The total number of Igusa parts is $2N_p -1$;  
\item if $p\equiv 13\ {\mathrm{mod}}\ 24$, there are three exceptional Igusa parts, of which two are copies
of $\overline{{\cal M}} ({\cal P}, {\mathrm{Ig}} (p)/ C_6 )_{\overline{\F}_p}$, and one is isomorphic to 
$\overline{{\cal M}} ({\cal P}, {\mathrm{Ig}} (p)/ C_4 )_{\overline{\F}_p}$. The total number of Igusa parts is $2N_p -1$;  
\item if $p\equiv 17\ {\mathrm{mod}}\ 24$, there are four exceptional Igusa parts, of which two are copies of 
$\overline{{\cal M}} ({\cal P}, {\mathrm{Ig}} (p)/ C_4 )_{\overline{\F}_p}$ and two are $\overline{{\cal M}} ({\cal P}, 
{\mathrm{Ig}} (p)/ C_8 )_{\overline{\F}_p}$. The total number of Igusa parts is $2N_p$;  
\item if $p\equiv 19\ {\mathrm{mod}}\ 24$, there are two exceptional Igusa parts, which are copies of $\overline{{\cal M}} 
({\cal P}, {\mathrm{Ig}} (p)/ C_6 )_{\overline{\F}_p}$. The total number of Igusa parts is $2N_p -1$;  
\item if $p\equiv 23\ {\mathrm{mod}}\ 24$, the total number of Igusa parts is $2N_p$, and none is exceptional.
\end{itemize} 

\medskip

  Singular points located on components of shape $\overline{{\cal M}} ({\cal P}, {\mathrm{Ig}} (p)/ C_r )_{\overline{\F}_p}$
 have local equations 
 $$
 W[[x,y]]/(xy-\pi_0^r ), r=1, 2, 3 {\mathrm{\ or\ }} 4.
 $$
%
%
%
\end{theo}

\begin{proof}

   One first readily checks that the isotropy groups $G_n$, $n=2, 3$ or $4$, are all cyclic, with order $n$. So
a generator for $G_n$ can be taken as $\left( {\zeta_n \atop 0}{0\atop 1} \right)$, for $\zeta_n$ some primitive
$n^{\mathrm{th}}$-root of unity in $\F_p$. 

   For $p\equiv 1, 7, 17$ {\rm mod} $24$, one computes at hand that the determinant of those generators 
are squares in  $\F_p$. (This again could also have been derived from the fact that  $\mathfrak{S}_4$ belongs to 
${\mathrm{SL}}_2 (\F_{p}) /\{ \pm 1\}$ if (and only) if $p\equiv \pm 1$ {\rm mod} $8$, cf.~Theorems~1 \& 2 in~\cite{La76}, 
pp.~201 \& 202).  Whence the vertical components, for primes in those congruences classes (and 
$p\equiv 23$ {\rm mod} $24$) in our theorem. 

   For the remaining classes we proceed with case-by-case examinations.   

If $p\equiv 5$ {\rm mod} $24$, a generator for the non-trivial isotropy group $G_4$ in $\mathfrak{S}_4$ can be taken as
$\left( {\zeta_4 \atop 0}{0\atop 1} \right)$, whose determinant is not a square in $\F_p$. So there is one 
exceptional Igusa part, which is a  copy of $\overline{{\cal M}} ({\cal P}, {\mathrm{Ig}} (p)/ C_4  )_{\overline{\F}_p}$.

If $p\equiv 11$ {\rm mod} $24$, a generator for the non-trivial isotropy group $G_2$ in $\mathfrak{S}_4$ can be taken as
$\left( {-1 \atop 0}{0\atop 1} \right)$, whose determinant is not a square in $\F_p$. So the corresponding Igusa part
is just a plain copy of $\overline{{\cal M}} ({\cal P}, {\mathrm{Ig}} (p)/\{ \pm 1 \} )_{\overline{\F}_p}$. 

If $p\equiv 13$ {\rm mod} $24$, the elements of $G_3$ in $\mathfrak{S}_4$ have square determinant,
so the corresponding orbits give rise to two exceptional Igusa parts which are  copies of $\overline{{\cal M}} ({\cal P}, 
{\mathrm{Ig}} (p)/ C_6 )_{\overline{\F}_p}$. On the other hand, the determinant of $\left( {\zeta_4 \atop 0}{0\atop 1} 
\right)$ is  a non-square in $\F_p$. So $G_4$ gives rise to a unique exceptional Igusa part, isomorphic to 
$\overline{{\cal M}} ({\cal P}, {\mathrm{Ig}} (p)/ C_4 )_{\overline{\F}_p}$.

If $p\equiv 19$ {\rm mod} $24$, the group $G_3$, in a similar fashion to the previous case, gives rise to two exceptional 
Igusa parts which are  copies of $\overline{{\cal M}} ({\cal P},  {\mathrm{Ig}} (p)/ C_6 )_{\overline{\F}_p}$. 
Similarly to the case $p\equiv 13$ {\rm mod} $24$, on the other hand, $G_2$ gives rise to one plain Igusa part
$\overline{{\cal M}} ({\cal P}, {\mathrm{Ig}} (p)/ \{ \pm 1\} )_{\overline{\F}_p}$. 
 
\medskip 
 
  The shape of singularities easily follows from our description of local isotropy groups.   \hspace{1cm} $\Box$

\end{proof} 
 
\bigskip
 
   We compute equations of Drinfeld components. Using Theorem~6.1 of~\cite{Ch99}, we can take as a system of 
generators for the image of ${\mathfrak{S}_4}$ in ${\mathrm{SL}}_2 (\C )/\{ \pm 1\}$ any set  $(S,T)$ whose traces satisfy
$$
t_S^2 +t_T^2 +t_{ST}^2 -t_S t_T t_{ST} =3 {\hspace{0.5cm}}  {\mathrm{and}}  {\hspace{0.5cm}}  {t_S},  {t_T},  {t_{ST}}
\in \{ 0, \pm 1, \pm \sqrt{2}  \} .
$$
One readily checks with that numerical criterion that ${\mathfrak{S}_4}$ has for instance a model in ${\mathrm{SL}}_2 
(\overline{\Z} )/\{ \pm 1\}$ with generators:
 $$
 S=
\left( 
\begin{array}{cc}
\sqrt{2}  &1\\
-1 & 0 
 \end{array}
\right) {\hspace{0.5cm}}  {\mathrm{and}}  {\hspace{0.5cm}}
T=
\left( \begin{array}{cc}
1  &i\\
i & 0 
 \end{array}
\right) 
$$
which when $p\equiv 1$ {\rm mod} $8$ gives by reduction an easy model in ${\mathrm{SL}}_2 (\F_p )/\{ \pm 1\}$.

\medskip

{\bf If $p=73$}, we can compute the orbit of $0$ and that of $1$, which respectively are
$$
A:= {\mathfrak{S}_4} \cdot 0 = \{ 0, 5, 16, 17, 26, 32, 39, 46, 52, 61, 62,  \infty   \}
$$
\begin{center}
and 
\end{center}
$$
B:= {\mathfrak{S}_4} \cdot 1 =\{ 1, 4, 6, 13, 18, 19, 23, 27, 28, 31, 33, 34, 36, 42, 44, 45, 47, 50, 51, 55, 59, 60, 65, 72   \}
$$
so that, setting $\phi (t) =\prod_{a\in A\setminus \{ \infty \}} (t-a )^2 \prod_{a\in B} (t-b )^{-1}$, we have
$\phi (\PPP^1 (\F_p ))= \{ 48, 14, 0, 58, \infty   \} $, whose elements have respective ramification indices
$4, 3, 2, 1$ and $1$. The inverse of the latter {\rm mod} $37$ are respectively $28, 25, 19, 1$ and $1$,  
whence the explicit forms 
$$
u^{37} =(t+25)^{28} (t-14)^{25} t^{19} (t+15)
$$ 
%
%
of equation~(\ref{equationD/H}) for the ${\mathfrak{S}_4}$-Drinfeld components in level $73$.

\subsection{${\mathfrak{A}_5}$}
\label{sectionA5}

One knows that, for any prime-power $q$, whenever ${\mathfrak{A}_5}$ can be realized as a subgroup of some ${\mathrm{GL}}_2 
(\F_q )/\F_q^*$ then it belongs to ${\mathrm{SL}}_2 (\F_q )/\{ \pm 1\}$, and that is the case if and only if $q$ is $\pm 1$ {\rm mod} $5$ (cf.~\cite{Fe76}, 
Theorem~1 on p.~201). Let us henceforth assume for that subsection~\ref{sectionA5} that $q=p$ is a prime satisfying that congruence 
condition. (We therefore remark that the smallest number field over which the corresponding modular curve has a geometrically integral model 
is the quadratic subfield of $\Q (\mu_p ))$. Again~\cite{La76} (proof of 
Theorem~2.3 on p.~186 of Chapter~XI) gives that there are three orbits of elements in $\PPP^1 (\F_{p^2} )$ with  
non-trivial isotropy subgroups for the action of $\mathfrak{A}_5$ in ${\mathrm{GL}}_2 (\F_{p^2}) /\F_{p^2}^*$, and those isotropy 
subgroups have order $2$, $3$ and $5$ (cf. case (v) of the Lemma after Theorem~2.3 quoted above): call them $G_2$,
$G_3$ and $G_5$. In $\PPP^1  (\F_{p^2} )$, there is therefore one orbit of size $30$, one of size
$20$, one of size $12$, and $(p^2  -61)/60$ orbits of size $60$. We denote by $O_{2}$, $O_3$ and $O_5$ the 
exceptional orbits of order $30$, $20$ and $12$ respectively. The combinatorics implies that, restricting to~$\PPP^1 (\F_p )$:
\medskip

\begin{center}

\begin{tabular}{|c||c|c|}
\hline $p$  {\rm mod} 60 & exceptional orbits  in $\PPP^1 (\F_p )$ &   total number $N_p$ of orbits under ${\mathfrak{A}}_5$ \\
\hline   
\hline
1 & $O_2$ , $O_3$, $O_5$ & $ (p+119) /60$ \\
\hline   
11 & $O_5$ & $(p+49)/60$ \\
\hline 
19 & $O_3$  & $(p+41)/60$ \\
\hline 
29 & $O_2$ & $(p+31)/60$ \\
\hline 
31 & $O_3$ , $O_5$ & $ (p+89) /60$ \\
\hline
41 & $O_2$, $O_5$ & $(p+79)/60$ \\
\hline 
49 & $O_2$, $O_3$ & $(p+71)/60$ \\
\hline 
59 & none & $(p+1)/60$  \\ 
\hline 
\end{tabular} .

\end{center}

\medskip

\begin{theo} 
\label{A5} 
Let $p>3$ be a prime, and let $[\Gamma_{\mathfrak{A}_5} (p)]$ be the moduli problem over $\Z [1/p]$ associated with 
$\Gamma_{\mathfrak{A}_5} (p)$.
Let ${\cal P}$ be a representable moduli problem, which is finite \'etale over $({\mathrm{Ell}})_{/\Z_p}$.  
Let $\overline{{\cal M}} ({\cal P}, \Gamma_{\mathfrak{A}_5} (p)) =\overline{{\cal M}} ({\cal P}, \Gamma (p))/ 
\Gamma_{\mathfrak{A}_5} (p)$ be the associated 
compactified fine moduli space. 
Let $W$ be a totally ramified extension of $\Z_p^{\mathrm{ur}}$ of degree $(p^2 -1)/2$ with uniformizer $\pi_0$, as in Theorem~\ref{NS}. 

\medskip

   Then $\overline{{\cal M}} ({\cal P}, \Gamma_{\mathfrak{A}_5} (p))$ has a semistable model over $W$ whose special 
fiber is made of $2N_p$ vertical Igusa parts (for $N_p \sim p/60$ as in the above array), which are linked by 
horizontal Drinfeld components above each supersingular points of $\overline{{\cal M}} ({\cal P})$ via the projection 
$\overline{{\cal M}} ({\cal P}, \Gamma_{\mathfrak{A}_5} (p)) \to \overline{{\cal M}} ({\cal P})$. Almost all
vertical parts
are geometrically isomorphic to quotient enhanced Igusa curves $\overline{{\cal M}} ({\cal P}, {\mathrm{Ig}} (p)/
\{ \pm1 \} )_{\overline{\F}_p} $, except that: 
\begin{itemize}
\item If $p\equiv 1$ {\rm mod} $60$,  two exceptional Igusa parts are $\overline{{\cal M}} ({\cal P}, {\mathrm{Ig}} (p)/C_4 
)_{\overline{\F}_p} $, two are $\overline{{\cal M}} ({\cal P}, {\mathrm{Ig}} (p)/ C_6  )_{\overline{\F}_p} $, and
two are $\overline{{\cal M}} ({\cal P}, {\mathrm{Ig}} (p)/ C_{10}  )_{\overline{\F}_p} $, for $C_*$ a cyclic automorphism 
group of order $*$. 
\item If $p\equiv 11$ {\rm mod} $60$, there are two exceptional Igusa parts, which are $\overline{{\cal M}} ({\cal P}, 
{\mathrm{Ig}} (p)/ C_{10}  )_{\overline{\F}_p} $;
\item If $p\equiv 19$ {\rm mod} $60$, two exceptional Igusa parts are copies of $\overline{{\cal M}} ({\cal P}, 
{\mathrm{Ig}} (p)/ C_6 )_{\overline{\F}_p} $;
\item If $p\equiv 29$ {\rm mod} $60$, two Igusa parts are copies of $\overline{{\cal M}} ({\cal P}, 
{\mathrm{Ig}} (p)/ C_4 )_{\overline{\F}_p} $;  
\item If $p\equiv 31$ {\rm mod} $60$, two Igusa parts are $\overline{{\cal M}} ({\cal P}, 
{\mathrm{Ig}} (p)/ C_6 )_{\overline{\F}_p} $ and two are  $\overline{{\cal M}} ({\cal P}, 
{\mathrm{Ig}} (p)/ C_{10} )_{\overline{\F}_p} $;  
 \item If $p\equiv 41$ {\rm mod} $60$, two Igusa parts are $\overline{{\cal M}} ({\cal P}, 
{\mathrm{Ig}} (p)/ C_4 )_{\overline{\F}_p} $ and two are  $\overline{{\cal M}} ({\cal P}, 
{\mathrm{Ig}} (p)/ C_{10} )_{\overline{\F}_p} $;  
 \item If $p\equiv 49$ {\rm mod} $60$, two Igusa parts are $\overline{{\cal M}} ({\cal P}, 
{\mathrm{Ig}} (p)/ C_4 )_{\overline{\F}_p} $ and two are  $\overline{{\cal M}} ({\cal P}, 
{\mathrm{Ig}} (p)/ C_{6} )_{\overline{\F}_p} $.
\end{itemize} 

\medskip

    Singular points located on components of shape $\overline{{\cal M}} ({\cal P}, {\mathrm{Ig}} (p)/ C_r )_{\overline{\F}_p}$
 have local equations 
 $$
 W[[x,y]]/(xy-\pi_0^r ), r=1, 2, 3 {\mathrm{\ or\ }} 5.
 $$
%
%
\end{theo}

\begin{proof}
As $\mathfrak{A}_5$ in fact belongs to ${\mathrm{SL}}_2 (\F_p )/\{ \pm 1\}$, the image under the determinant of the full 
group $\Gamma_{\mathfrak{A}_5} (p)$ 
consists in all the squares of $\F_p^*$. The quotient of the set of vertical Igusa components, indexed by $\F_p^* 
\times \PPP^1 (\F_p )$, therefore has twice the number of elements as indicated in the list above, depending on the class
of $p$ {\rm mod} $60$.  

Igusa components associated with orbits of size $n=60, 30, 20$ or $12$, have stabilizer $C_{(120/n)}$ in 
$\Gamma_{\mathfrak{A}_5} (p)$ so they are isomorphic to $\overline{{\cal M}} ({\cal P}, {\mathrm{Ig}} (p)/
C_{(120/n)} )_{\overline{\F}_p }$. The rest goes as in the proof of the previous theorems. $\Box$

\end{proof}

\bigskip 

As for Drinfeld components: using Theorem~6.1 of~\cite{Ch99}, we can take as a system of 
generators for the image of ${\mathfrak{A}_5}$ in ${\mathrm{SL}}_2 (\overline{\Z} )/\{ \pm 1\}$ any set  $(S,T)$ whose traces satisfy
$$
t_S^2 +t_T^2 +t_{ST}^2 -t_S t_T t_{ST} \in \{ 2+\mu , 3, 2-\mu^{-1} \} {\hspace{0.2cm}}  {\mathrm{and}}  {\hspace{0.2cm}}  {t_S},  {t_T},  {t_{ST}}
\in \{ 0, \pm \mu , \pm 1, \pm \mu^{-1} \}, {\hspace{0.2cm}}  {\mathrm{for}}  {\hspace{0.2cm}} \mu = \frac{1+\sqrt{5}}{2} .
$$
{\bf Taking $p=421$}, we can choose $S$ and $T$ as the reduction of the generators displayed in the introduction to~\cite{Ch99}, 
that is
$$
S=
\left( 
\begin{array}{cc}
211 & 196 \\
316 & 100
 \end{array}
\right) {\hspace{0.5cm}}  {\mathrm{and}}  {\hspace{0.5cm}}
T=
\left( \begin{array}{cc}
100  & 70 \\
306 & 210 
 \end{array}
\right) 
$$  
in ${\mathrm{SL}}_2 (\F_{421} )/\{ \pm 1\}$. Drawing the graph of the homographic action of $S$ and $T$ on 
the elements $0$ and $1$ in $\PPP^1 (\F_{421})$ yields the respective orbits

\medskip

\noindent $A:= {\mathfrak{A}_5} \cdot 0 = \{$0, 2, 3, 14, 17, 20, 29, 50, 51, 55, 72, 83, 94, 101, 146, 152, 153, 156, 163, 166, 177, 182, 
190, 191, 192, 203, 206, 209, 210,   211, 212, 215, 218, 220, 222, 225,  230, 234, 236, 242, 250, 257, 264, 266, 279,  284, 293,   319, 326, 335,  343, 
352,    355,     357, 359,     392,   396, 418, 419, $\infty \}$

\medskip

\begin{center}
and 
\end{center}

\medskip

\noindent
$B:= {\mathfrak{A}_5} \cdot 1 =\{$1, 5, 23, 25, 26, 27, 35, 40, 60, 61, 81, 92, 93, 105, 107, 115, 127, 128, 137, 143, 154, 159, 
160, 164, 172, 173, 189, 193, 195, 202, 223, 227, 233, 235, 243, 246, 252, 256, 259,  273, 274, 289, 294, 306, 323, 324, 325, 327, 348, 350, 361, 363, 370, 373, 
374, 379, 382, 388, 389, 409$\}$. 

\medskip

Setting $\phi (t) =\prod_{a\in A\setminus \{ \infty \}} (t-a )^2 \prod_{a\in B} (t-b )^{-1}$ one computes
$$
\phi (\PPP^1 (\F_{421} ))=\{ 0, 23, 47, 144, 161, 228, 292, 317, \infty  \} 
$$
with ramification indices: $1, 2, 1, 3, 1, 1, 1, 5$ and $1$ respectively. The list of their inverse {\rm mod} $211$ is $1, 106, 1,  141,  1, 1, 1,  169$ and $1$,
so that an equation as in~(\ref{equationD/H}) for the generic ${\mathfrak{A}_5}$-Drinfeld component in level $421$ is finally
$$
u^{211} =t (t-23)^{106} (t-47) (t-144)^{141} (t-161) (t-228) (t-292) (t-317)^{169}  .
$$

\bigskip

\noindent
Bas Edixhoven\\
Mathematisch Instituut, Universiteit Leiden, Niels Bohrweg 1, 2333 CA Leiden, Nederland\\
edix@math.leidenuniv.nl

\bigskip

\noindent 
Pierre Parent\\
Univ. Bordeaux, CNRS, Bordeaux INP, IMB, UMR 5251,  F-33400, Talence, France\\
Pierre.Parent@math.u-bordeaux.fr

\end{document}